\documentclass[11pt]{article}

\usepackage[T1]{fontenc}
\usepackage{microtype}
\usepackage{amsmath,amssymb,amsthm,mathtools}
\usepackage{newtxtext,newtxmath}
\usepackage[a4paper,margin=27mm]{geometry}
\usepackage{booktabs,tabularx,array}
\usepackage{aliascnt}
\usepackage{enumitem}
\usepackage[numbers,sort&compress]{natbib}
\usepackage[colorlinks=false,hidelinks]{hyperref}
\usepackage[capitalise,noabbrev]{cleveref}
\usepackage{fancyhdr}

\setlist[itemize]{leftmargin=1.6em,itemsep=0.15em,topsep=0.3em}
\setlist[enumerate]{leftmargin=1.9em,itemsep=0.15em,topsep=0.3em}
\allowdisplaybreaks[2]
\usepackage{subcaption}

\newtheorem{theorem}{Theorem}[section]
\newaliascnt{lemma}{theorem}
\newtheorem{lemma}[lemma]{Lemma}
\aliascntresetthe{lemma}
\newaliascnt{proposition}{theorem}
\newtheorem{proposition}[proposition]{Proposition}
\aliascntresetthe{proposition}
\newaliascnt{corollary}{theorem}
\newtheorem{corollary}[corollary]{Corollary}
\aliascntresetthe{corollary}
\theoremstyle{definition}
\newaliascnt{definition}{theorem}
\newtheorem{definition}[definition]{Definition}
\aliascntresetthe{definition}
\newaliascnt{remark}{theorem}

\aliascntresetthe{remark}
\newaliascnt{example}{theorem}

\aliascntresetthe{example}
\newaliascnt{observation}{theorem}

\aliascntresetthe{observation}

\newcommand{\cH}{\mathcal H}
\newcommand{\cA}{\mathcal A}
\newcommand{\cC}{\mathcal C}
\newcommand{\cM}{\mathcal M}

\newcommand{\Cl}{\operatorname{cl}}
\newcommand{\Stop}{\operatorname{Stop}}
\newcommand{\MinStop}{\operatorname{MinStop}}
\newcommand{\Tr}{\operatorname{Tr}}
\newcommand{\False}{\operatorname{False}}
\newcommand{\rank}{\operatorname{rank}}

\newcommand{\MaxMod}{\operatorname{MaxMod}}
\newcommand{\Mod}{\operatorname{Mod}}
\newcommand{\HHEnum}{\textsc{HH-Coatom-Enum}}
\newcommand{\HHEx}{\textsc{HH-Coatom-Extension}}
\newcommand{\IncHH}{\textsc{Incremental HH-Coatom}}

\newcommand{\OutputP}{\mathsf{OutputP}}
\newcommand{\DelayP}{\mathsf{DelayP}}
\newcommand{\NP}{\mathsf{NP}}

\newcommand{\one}{\mathbf 1}

\crefname{theorem}{Theorem}{Theorems}
\crefname{lemma}{Lemma}{Lemmas}
\crefname{proposition}{Proposition}{Propositions}
\crefname{corollary}{Corollary}{Corollaries}
\crefname{definition}{Definition}{Definitions}
\crefname{remark}{Remark}{Remarks}
\crefname{example}{Example}{Examples}
\crefname{observation}{Observation}{Observations}

\theoremstyle{definition}
\newtheorem{problem}{Problem}[section]

\crefname{problem}{problem}{problems}
\Crefname{problem}{Problem}{Problems}

\usepackage{tikz}
\usetikzlibrary{
  arrows.meta,
  backgrounds,
  fit,
  positioning,calc
}
\usepackage{tikz-cd}

\tikzset{
  stopvar/.style={
    circle,
    draw, thick,
    minimum size=5.2mm,
    inner sep=0pt,
    font=\small
  },
  stopedge/.style={
    rectangle,
    draw,
    rounded corners=0.8pt,
    minimum size=4.6mm,
    inner sep=0.5pt,
    font=\scriptsize
  },
  compilerblock/.style={
    draw,
    rounded corners,
    align=center,
    inner sep=4pt,
    font=\small,
    fill=black!3
  },
  groupbox/.style={
    draw,
    dashed,
    rounded corners,
    inner sep=6pt
  },
    primitivepanel/.style={
    draw, dashed, rounded corners=3pt, inner sep=6pt
  },
  hedgeA/.style={
    draw, rounded corners=6pt, inner sep=3pt
  },
  hedgeB/.style={
    draw, dashed, rounded corners=6pt, inner sep=3pt
  },
  hedgeC/.style={
    draw, densely dotted, rounded corners=6pt, inner sep=3pt
  }
}

\title{Coatom Enumeration in Hypergraph Horn Functions:
Rank-Three Representations of Horn Model Posets}
\author{
Jianshen Zhu$^1$}
\date{
 $^1$Department of Information Sciences and Technology, Tokyo University of Science, Noda, Chiba 278-8510, Japan \\
}

\begin{document} 
\maketitle
\begin{abstract}
For a finite hypergraph $\cH$, the complements of the models of its
associated definite Horn CNF are exactly the stopping sets of
$\cH$; hence the complements of its coatoms are the inclusion-minimal
nonempty stopping sets.
We study their output-sensitive enumeration from the hypergraph
incidence representation.
Our main result is a representation of arbitrary Horn model posets 
whose incidence size is linear in the incidence length of the normalized Horn input.
Given a Horn CNF $\Gamma$, we construct a hypergraph 
$\cC(\Gamma)$ of rank at most three whose proper-model poset 
is inclusion-order isomorphic to the model poset of $\Gamma$; 
equivalently, each source model has a unique extension to a proper target model.
Thus maximal models of $\Gamma$ correspond bijectively to target
coatoms.
Combining this representation with the maximal-model construction of
Kavvadias, Sideri, and Stavropoulos shows that coatom enumeration is
not in $\OutputP$ unless $\mathsf P=\mathsf{NP}$, even when all
hyperedges have size two or three.
Incidence splitting reduces maximum element frequency to three while
preserving the stopping-set poset, and a local replacement of
two-element hyperedges yields the same lower bound for three-uniform
hypergraphs of maximum element frequency at most three.
These thresholds are conditionally sharp for arbitrary-order enumeration: rank at 
most two and maximum element frequency at most two both admit output-linear total-time generation; 
in the frequency-two case, a polynomial-delay, polynomial-space algorithm is also available.
In contrast, coatom extension is $\NP$-complete already for
three-uniform hypergraphs of exact element frequency two.
\end{abstract}


\section{Introduction}\label{sec:introduction}

Hypergraph Horn functions, introduced and characterized by
\citet{BercziBorosMakino2024}, are definite Horn functions that admit
a representation by a finite hypergraph.
More precisely, a hypergraph $\cH\subseteq 2^V$ represents the
definite Horn CNF
\begin{equation}
  \Phi_{\cH}
  =
  \bigwedge_{A\in\cH}\ \bigwedge_{v\in A}
  \bigl((A\setminus\{v\})\to v\bigr).
  \label{eq:circular-formula1}
\end{equation}
Thus each hyperedge $A$ encodes the family of implications in which
every element of $A$ is implied by all the others.
The models of $\Phi_{\cH}$ form an intersection-closed family and
therefore define a closure system on $V$.

In this paper, we study the coatoms of this closure system, namely, the closed sets
that are maximal among the proper closed sets.
In database terminology these are antikeys, whose relation with
minimal keys is classical
\citep{Thi1986,DemetrovicsThi1987,Son2006}.
Their complements admit an especially simple description.
For $M\subseteq V$ and $S=V\setminus M$, the set $M$ is closed if and only if
\begin{equation}
  |A\cap S|\neq 1
  \qquad\text{for every }A\in\cH.
  \label{eq:stopping-condition-intro}
\end{equation}
Consequently, the complements of the coatoms are precisely the
inclusion-minimal nonempty sets satisfying
\eqref{eq:stopping-condition-intro}.

Condition~\eqref{eq:stopping-condition-intro} is also the standard
stopping-set condition from coding theory.
Given a binary parity-check matrix $H$, take its columns as the ground
set and the support of each row as a hyperedge.
A set $S$ is then a stopping set of the Tanner graph of $H$ exactly
when no row support intersects $S$ in a single coordinate.
On the binary erasure channel, such sets characterize residual
erasure patterns on which iterative peeling decoding may fail
\citep{DiEtAl2002,RichardsonUrbanke2008}.
Their role in finite-length performance has motivated the study of
stopping distance, stopping redundancy, and parity-check matrix
extensions that eliminate small stopping sets
\citep{SchwartzVardy2006,HanSiegel2007,FalsafainMousavi2015}.
Algorithms have also been developed for finding all stopping sets up
to a prescribed size in LDPC matrices
\citep{RosnesYtrehus2009,WangKulkarniPoor2009,RosnesEtAl2012}.
These algorithms address bounded-size, code-specific search, whereas
we study the output-sensitive generation of the complete family of
inclusion-minimal nonempty stopping sets from an arbitrary incidence
representation.

We call the resulting enumeration problem \HHEnum. 
The restriction to inclusion-minimal stopping sets is natural from the closure-system viewpoint: 
by \cref{prop:complement-duality}, they are exactly the complements of the 
coatoms of $\Phi_{\cH}$. This output family behaves differently from 
the implicate sets studied by \citet{BercziBorosMakino2024}, 
which can be generated with polynomial delay. We show that coatom enumeration 
is not output-polynomial unless $\mathsf P=\mathsf{NP}$, even under strong rank and frequency restrictions.

The problem also belongs to the broader literature on translations
between compact representations of closure systems.
General relations among lattices, closure systems, and implicational
bases are surveyed by \citet{BertetEtAl2018}.
In Horn logic, meet-irreducible closed sets form the characteristic
model representation studied by \citet{Khardon1995}.
Every coatom of a finite closure lattice is meet-irreducible, but the
converse need not hold; hence enumerating all meet-irreducible closed
sets is a broader output problem than the one considered here.
Related algorithmic results have been obtained for $k$-meet-semidistributive 
lattices \citep{BeaudouMaryNourine2017}, for implicational bases 
admitting suitable hierarchical decompositions \citep{NourineVilmin2023}, 
and for acyclic implicational bases of bounded degree \citep{DefrainOhanaVilmin2025}.

Lattice dualization from an implicational basis is another related
task.
\citet{DefrainNourine2020} proved that it is not output-polynomial
unless $\mathsf P=\mathsf{NP}$ even when all implication premises
have size at most two.
These results do not subsume the present ones.
Lattice dualization takes antichains in an implicationally represented
lattice as part of its input, whereas we enumerate the coatoms
directly from a hypergraph Horn incidence representation.
Moreover, bounds on premise size or degree in an arbitrary 
implicational basis do not directly translate into bounds on rank or 
element frequency in the hypergraph Horn incidence representation 
considered here. Our contribution is a representation of the entire 
Horn model poset by a rank-at-most-three hypergraph Horn function 
whose incidence length is linear in that of the input Horn CNF 
and which gives each source model a unique target extension, 
together with coatom-specific rank and frequency thresholds.

For the lower bound developed here, we use maximal-model enumeration
for Horn CNFs.
\citet{KavvadiasSideriStavropoulos2000} proved that maximal Horn models 
cannot be enumerated in output-polynomial time unless $\mathsf P=\mathsf{NP}$. 
Their result does not transfer directly, because an arbitrary Horn CNF need not admit a representation 
by the symmetric all-but-one implication bundles in \eqref{eq:circular-formula1}. 
A reduction must therefore preserve the maximal-model family while converting the formula into a hypergraph Horn representation.

Our main representation theorem provides such a conversion. 
Given a Horn CNF $\Gamma$ on a variable set $X$, it constructs a hypergraph $\cC(\Gamma)$ 
of rank at most three such that the model poset of $\Gamma$ is inclusion-order isomorphic 
to the proper-model poset of $\Phi_{\cC(\Gamma)}$. Equivalently, if $M$ is a model of $\Gamma$, 
then its false-variable set $X\setminus M$ has a unique extension to a nonempty stopping set of $\cC(\Gamma)$; 
every nonempty stopping set arises in this way, and the correspondence preserves and reflects inclusion. 
Consequently, maximal models of $\Gamma$ correspond bijectively to target coatoms, 
and the auxiliary elements introduce no additional coatoms. 
The construction has linear incidence size and polynomial-time forward and inverse maps.
As a structural consequence, every finite closure system presented by a Horn basis 
is isomorphic to the proper-model poset of a rank-at-most-three hypergraph Horn function. 
Hence every finite lattice admits such a representation through its lattice of principal ideals; 
see \cref{cor:rank3-universality}. This representation uses auxiliary elements and 
does not assert that the original Horn function is itself hypergraph Horn.

Combining the representation theorem with the hard Horn formulas 
of \citet{KavvadiasSideriStavropoulos2000} shows that $\textsc{HH-Coatom-Enum}$ 
is not in $\OutputP$ unless $\mathsf P=\mathsf{NP}$, even when every hyperedge 
has size two or three. The compiler controls hyperedge size but not element frequency. 
We therefore apply an incidence-splitting transformation that replaces repeated 
occurrences by equality-linked copies. This transformation preserves the entire 
stopping-set poset and reduces maximum element frequency to three.
To remove the remaining two-element hyperedges, we replace each of them 
by four three-element hyperedges. The replacement preserves every original minimal 
stopping set and introduces exactly three explicitly known auxiliary minimal stopping 
sets for each replaced hyperedge. Removing these auxiliary outputs transfers 
the lower bound to three-uniform hypergraphs of maximum element frequency at most three.

Under the parity-check interpretation, hypergraph rank is the maximum check-node degree, 
or maximum row weight, and element frequency is the maximum variable-node degree, 
or maximum column weight. The final hard instances therefore have check-node 
degree exactly three and variable-node degree at most three.

The neighboring positive cases make these bounds conditionally sharp for arbitrary-order enumeration. 
If every hyperedge has size at most two, the closure system decomposes into equality components, 
and its coatoms can be generated in output-linear total time after normalization. 
If every element has frequency at most two, the elements can be viewed as edges 
of an auxiliary multigraph whose vertices are the hyperedges; the inclusion-minimal stopping 
sets then correspond to graphic circuits. The optimal cycle-listing algorithm of \citet{BirmeleEtAl2013} 
yields output-linear total time, while a direct implementation gives polynomial delay 
and polynomial working space. Thus both rank two and frequency two admit output-linear 
total-time enumeration, whereas frequency three already permits a conditional non-$\OutputP$ lower bound.
The extension problem has a different threshold. It remains $\NP$-complete for three-uniform hypergraphs 
in which every element has frequency exactly two, equivalently for Tanner graphs with check-node 
degree three and variable-node degree two. 
These degree restrictions concern the given incidence or parity-check representation. 
Hypergraphs representing the same Horn function have the same stopping-set family, 
whereas different parity-check matrices for the same code may define different Horn functions 
and hence different stopping-set families. 
The results are summarized in \cref{tab:thresholds}.

\begin{table}[t]
\centering
\caption{Complexity at the rank and frequency thresholds.}
\label{tab:thresholds}
\small
\begin{tabularx}{\textwidth}{
  @{}
  >{\raggedright\arraybackslash}p{0.25\textwidth}
  >{\raggedright\arraybackslash}X
  >{\raggedright\arraybackslash}p{0.30\textwidth}
  @{}
}
\toprule
Input class
  & Arbitrary-order enumeration
  & Extension \\
\midrule
$\rank(\cH)\le2$
  & output-linear after normalization; polynomial space
  & in $\mathsf P$ \\[1mm]

$\Delta(\cH)\le2$
  & output-linear total time; also in $\DelayP$ with polynomial space
  & $\NP$-complete even for $3$-uniform inputs with
    $\Delta(\cH)=2$ \\[1mm]

$\cH$ is $3$-uniform and $\Delta(\cH)\le3$
  & not in $\OutputP$ unless $\mathsf P=\mathsf{NP}$
  & $\NP$-complete already for $3$-uniform inputs in which every
    element has frequency exactly two \\ \\
\bottomrule
\end{tabularx}
\end{table}

These results complete the classification of \HHEnum\ under
arbitrary-order enumeration for input classes defined only by fixed
upper bounds on rank and maximum element frequency.
If the rank bound or the frequency bound is at most two, 
the corresponding output-sensitive algorithms apply. 
If both bounds are at least three, the three-uniform hard instances show 
that the problem is not in $\OutputP$ unless $\mathsf P=\mathsf{NP}$.
The complexity also depends on the task imposed on the coatom family. 
One arbitrary coatom can be found in polynomial time, whereas minimizing 
its complement is $\NP$-hard. 
Constrained extension and deciding whether a given 
explicit list consists of coatoms and omits another coatom are $\NP$-complete. 
The latter remains hard even when every listed set is promised to be a coatom.
Complete arbitrary-order enumeration 
is not in $\OutputP$ unless $\mathsf P=\mathsf{NP}$.
As an additional comparison with classical hypergraph dualization, 
we give a direct linear-size, output-bijective reduction from 
explicit minimal-transversal enumeration to $\textsc{HH-Coatom-Enum}$. 
The construction maps every source transversal to exactly one nonempty 
stopping set and introduces no additional nonempty stopping sets. 
It does not preserve bounded rank or bounded frequency, so it is separate from the threshold reductions above. 
Its proof, exact output maps, and boundary cases are given in Appendix~\ref{app:transversal}.

The rest of the paper is organized as follows. 
\Cref{sec:prelim} introduces the hypergraph Horn representation, 
stopping-set duality, and the output-sensitive complexity model.
\Cref{sec:realization} develops the representation construction 
and proves the Horn-model-poset theorem. 
\Cref{sec:consequences} transfers the maximal-model lower 
bounds to coatom enumeration, the incremental problem, and constrained extension.
\Cref{sec:thresholds} proves the rank-two and frequency-two positive results, 
reduces maximum element frequency to three, and establishes the three-uniform lower bounds. 
Finally, \Cref{sec:conclusion} concludes the paper. 
Appendix~\ref{app:kss} verifies the Kavvadias--Sideri--Stavropoulos reduction, 
Appendix~\ref{app:freq2-extension} supplies the triangle expansion used 
for the exact-frequency-two extension lower bound, 
Appendix~\ref{app:cycle-delay} gives the direct polynomial-delay implementation 
for the frequency-two enumerator, and 
Appendix~\ref{app:transversal} gives a linear-size, 
output-bijective embedding of minimal-transversal enumeration into \HHEnum{}.

\section{Preliminaries}
\label{sec:prelim}

This section fixes the notation and complexity conventions used
throughout the paper.  We first recall Horn formulas and their
hypergraph Horn representations.  We then introduce stopping sets and
their complement duality with closed sets, before formally defining
the enumeration and extension problems considered later.

\subsection{Horn formulas and hypergraph Horn functions}
\label{sec:horn-hypergraph-prelim}

Let $X$ be a finite set of Boolean variables.  We identify a truth
assignment with the set $M\subseteq X$ of variables assigned true.
A Horn clause is written either as a definite implication
$A\to b$, where $A\subseteq X$ and $b\in X\setminus A$, or as a
negative implication $A\to\bot$.
The assignment $M$ satisfies $A\to b$ if $A\subseteq M$ implies
$b\in M$, and it satisfies $A\to\bot$ if $A\nsubseteq M$.
A Horn conjunctive normal form (CNF) is a conjunction of Horn clauses.
An assignment is a \emph{model} of a Horn CNF $\Gamma$ if it satisfies
every clause of $\Gamma$.
We denote its model family by
\(
  \Mod(\Gamma)=\{M\subseteq X:M\models\Gamma\}.
\)
A model $M$ is \emph{maximal} if no model of $\Gamma$ properly
contains it.

We regard premises as sets and remove duplicate clauses.
Empty premises are permitted.
Thus $\varnothing\to b$ forces $b$ to be true, whereas
$\varnothing\to\bot$ makes the formula inconsistent.
For a normalized Horn CNF, we measure the input size by its incidence
length
\[
  \|\Gamma\|_{\mathrm{inc}}
  =
  |X|+|\Gamma|
  +\sum_{A\to b\in\Gamma}(|A|+1)
  +\sum_{A\to\bot\in\Gamma}|A|.
\]
A raw incidence-list input can be converted to this form in polynomial
time.

The target class studied in this paper consists of hypergraph Horn
functions, introduced and characterized by
\citet{BercziBorosMakino2024}.
A definite Horn function on a finite ground set $V$ is hypergraph Horn
if it admits a representation by a hypergraph
$\cH\subseteq 2^V$ through the CNF
\begin{equation}
  \Phi_{\cH}
  =
  \bigwedge_{A\in\cH}\ \bigwedge_{v\in A}
  \bigl((A\setminus\{v\})\to v\bigr).
  \label{eq:circular-formula}
\end{equation}
Thus each hyperedge contributes the family of implications in which
every one of its elements is implied in turn by all the others.

We take the hypergraph incidence list, rather than the explicitly
expanded CNF, as the input representation.
Before measuring its size, we remove repeated incidences within an
edge, duplicate hyperedges, and empty hyperedges.
The last may be discarded because an empty hyperedge contributes no
implication to \eqref{eq:circular-formula}.
For a normalized hypergraph, let
\(
  L(\cH)=|V|+|\cH|+\sum_{A\in\cH}|A|
\)
be its incidence length.
Its rank and maximum element frequency are
\[
  \rank(\cH)=\max_{A\in\cH}|A|,
  \qquad
  \Delta(\cH)
  =
  \max_{v\in V}|\{A\in\cH:v\in A\}|,
\]
where an empty maximum is understood to be zero.
In particular, frequency is measured after duplicate hyperedges have
been removed.
A raw incidence list of length $\widehat L$ can be normalized
deterministically in $O(\widehat L\log\widehat L)$ time by sorting
canonical edge encodings.

Since $\Phi_{\cH}$ is a definite Horn CNF, its model family is closed
under intersection and contains $V$.
It therefore forms a closure system.
For $Y\subseteq V$, the associated closure is
\[
  \Cl_{\cH}(Y)
  =
  \bigcap\{M\in\Mod(\Phi_{\cH}):Y\subseteq M\}.
\]
Equivalently, $\Cl_{\cH}(Y)$ is obtained by forward chaining from $Y$
with the implications in \eqref{eq:circular-formula}.
This computation can be carried out directly on the hypergraph
incidence representation.
For each hyperedge, one maintains the number of its elements not yet
present; when exactly one remains, that element is inserted.
Each vertex is inserted at most once and each incidence is processed
only constantly many times, so the computation takes
$O(L(\cH))$ time.
It uses $O(|V|+|\cH|)$ additional space once the incidence adjacency
lists have been constructed.

A \emph{coatom} of the closure system $\Mod(\Phi_{\cH})$ is a proper closed set that is maximal among the
proper closed sets.
We denote the family of coatoms by
\[
  \cM(\Phi_{\cH})
  =
  \max_{\subseteq}
  \bigl(\Mod(\Phi_{\cH})\setminus\{V\}\bigr).
\]
This family may be empty.

\subsection{Stopping sets and complement duality}
\label{sec:stopping-prelim}

To describe the complements of closed sets, let
$M\subseteq V$ and put $S=V\setminus M$.
For a hyperedge $A\in\cH$, the implication bundle contributed by $A$
is violated by $M$ precisely when exactly one element of $A$ lies
outside $M$, or equivalently when $|A\cap S|=1$.
Thus the complements of closed sets are characterized by the
condition that no hyperedge meets them in exactly one element.

When the hyperedges are the row supports of a binary parity-check
matrix, this is the standard stopping-set condition for the associated
Tanner graph
\citep{DiEtAl2002,RichardsonUrbanke2008}.
We adopt the same terminology for an arbitrary hypergraph.

\begin{definition}[Stopping set]
\label{def:stopping-set}
A set $S\subseteq V$ is a \emph{stopping set} of $\cH$ if
$|A\cap S|\ne1$ for every $A\in\cH$.
We denote the family of all stopping sets by
\[
  \Stop(\cH)
  =
  \{S\subseteq V:|A\cap S|\ne1
    \text{ for every }A\in\cH\}.
\]
Its inclusion-minimal nonempty members are called
\emph{minimal stopping sets}, and their family is denoted by
\[
  \MinStop(\cH)
  =
  \min_{\subseteq}
  \bigl(\Stop(\cH)\setminus\{\varnothing\}\bigr).
\]
\end{definition}

The preceding observation gives the basic duality used throughout the
paper.

\begin{proposition}
\label{prop:complement-duality}
For $M\subseteq V$ and $S=V\setminus M$,
\[
  M\models\Phi_{\cH}
  \quad\Longleftrightarrow\quad
  S\in\Stop(\cH).
\]
Hence complementation is an order-reversing bijection from
$\Mod(\Phi_{\cH})$ onto $\Stop(\cH)$, and
\[
  \MinStop(\cH)
  =
  \{V\setminus M:M\in\cM(\Phi_{\cH})\}.
\]
\end{proposition}

\begin{proof}
For $A\in\cH$ and $v\in A$, the implication
$(A\setminus\{v\})\to v$ is violated by $M$ precisely when
$A\setminus\{v\}\subseteq M$ and $v\notin M$.
With $S=V\setminus M$, this is equivalent to $A\cap S=\{v\}$.
Thus all implications contributed by $A$ are satisfied exactly when
$|A\cap S|\ne1$.

The remaining statements follow because complementation reverses
inclusion and maps the top closed set $V$ to the empty stopping set.
\end{proof}

Thus enumerating coatoms and enumerating minimal stopping sets are
equivalent up to complementation.
We state the computational problems in terms of coatoms and use
stopping sets in constructions and reductions.

\subsection{Computational problems and complexity conventions}
\label{sec:enumeration-prelim}

We now formulate the
three computational problems studied in this paper and fix the
output-sensitive complexity conventions used below.
All hypergraphs in the following problems are given by their
normalized incidence lists.  Additional subsets of the ground set are
encoded by characteristic vectors unless stated otherwise.

\begin{problem}[\HHEnum{}]
\label{prob:hh-enum}
Given a normalized hypergraph $\cH$ on $V$, enumerate all coatoms of
the closure system $\Mod(\Phi_{\cH})$, each exactly once and in
arbitrary order.
\end{problem}

By \cref{prop:complement-duality}, this is equivalent to enumerating
all members of $\MinStop(\cH)$ and complementing the resulting sets.

\begin{problem}[\HHEx{}]
\label{prob:hh-extension}
Given a normalized hypergraph $\cH$ on $V$ and disjoint sets
$P,N\subseteq V$, decide whether $\Mod(\Phi_{\cH})$ has a coatom $M$
such that
\[
  P\subseteq M
  \qquad\text{and}\qquad
  M\cap N=\varnothing.
\]
\end{problem}

In stopping-set coordinates, the same problem asks whether there is
an $S\in\MinStop(\cH)$ such that
\(
  N\subseteq S\subseteq V\setminus P.
\)

\begin{problem}[\IncHH{}]
\label{prob:incremental-hh}
Given a normalized hypergraph $\cH$ on $V$ and a list
$\mathcal L=(M_1,\ldots,M_t)$ of pairwise distinct subsets of $V$,
decide whether every $M_i$ is a coatom of
$\Mod(\Phi_{\cH})$ and there exists a coatom that does not appear in
$\mathcal L$.
\end{problem}

Thus a list containing a noncoatom is a no-instance.
We also consider the promise version in which all members of
$\mathcal L$ are guaranteed to be coatoms.

Unless stated otherwise, each coatom is output as its $|V|$-bit
characteristic vector.
Some algorithms are more naturally described in terms of the
complementary minimal stopping set $S=V\setminus M$; whenever a sparse
encoding of $S$ is used instead, this is stated explicitly.
The time needed to write an output is included in the running time.

An enumeration problem belongs to $\OutputP$ if it admits an
algorithm whose total running time is bounded by a polynomial in the
input length and the total encoded length of all outputs.
We call an enumeration algorithm output-linear if its total running time is linear in the input length plus the total encoded length of all outputs.
An algorithm has \emph{polynomial delay} if the preprocessing time
before the first output, the time between consecutive outputs, and
the time from the last output to termination are each bounded by a
polynomial in the input length.
We write $\DelayP$ for the class of enumeration problems admitting
such an algorithm.
Unless an order is explicitly prescribed, outputs may be generated
in arbitrary order.
A polynomial-space bound refers to working memory and excludes the
write-only output stream.

\section{A rank-at-most-three hypergraph representation of Horn model posets}
\label{sec:realization}

We give the main representation theorem in this section.
Starting from an arbitrary Horn CNF, we construct a hypergraph of
rank at most three whose nonempty stopping sets encode the models of
the source formula.
The construction has
linear incidence size, uniquely determines all auxiliary coordinates,
and preserves the inclusion order on the whole model family.

\subsection{Local stopping-set constructions}
\label{sec:local-primitives}

This subsection studies small hyperedge families that realize Boolean
relations through the stopping-set condition.
The three relations considered here are equality $p=q$, implication
$x\to y$, and binary disjunction $c=a\lor b$.
The elements corresponding to the variables of the represented
relation are called \emph{interface elements}.
An auxiliary element used only in one occurrence of a construction is
called \emph{private}; private elements belonging to distinct
occurrences are always chosen to be distinct.

Formally, let $\mathcal E\subseteq 2^Z$ be a hypergraph on a finite
set $Z$.
A Boolean assignment $\alpha\colon Z\to\{0,1\}$ is
\emph{feasible for $\mathcal E$} if
$\sum_{z\in A}\alpha(z)\ne1$ for every $A\in\mathcal E$, or
equivalently, if the support of $\alpha$ is a stopping set of
$\mathcal E$.
We say that $\mathcal E$ \emph{realizes} a Boolean relation on its
interface elements if an assignment to those elements belongs to the
relation exactly when it extends to a feasible assignment on $Z$.
The realization has \emph{unique extension} if this feasible extension
is unique.
When an assignment is fixed, we use the same symbol for a ground
element and its Boolean value.

The equality construction consists of the single hyperedge
$\{p,q\}$.
For interface elements $x,y$ and a private element $q$, define
\(
  \operatorname{Imp}(x,y;q)
  :=
  \bigl\{\{y,q\},\{x,y,q\}\bigr\}.
\)
For input elements $a,b$, an output element $c$, and a fresh private
element $q$, define
\(
  \operatorname{OR}(a,b;c,q)
  :=
  \bigl\{\{c,q\},\{a,b,c\},\{a,c,q\}\bigr\}.
\)
Here the symbol $q$ denotes a different private element in each
occurrence.
The three constructions are depicted in
\cref{fig:local-primitives}.

\begin{figure}[t]
\centering

\begin{subfigure}[b]{0.25\textwidth}
  \centering
  \includegraphics[
    height=3.05cm,
    width=\linewidth,
    keepaspectratio
  ]{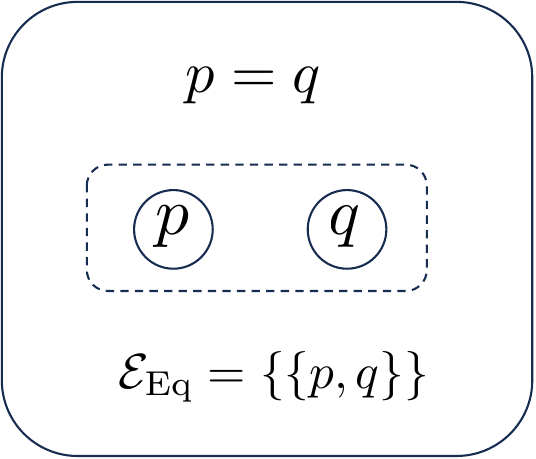}
  \caption{Equality}
  \label{fig:primitive-equality}
\end{subfigure}
\hfill
\begin{subfigure}[b]{0.27\textwidth}
  \centering
  \includegraphics[
    height=3.05cm,
    width=\linewidth,
    keepaspectratio
  ]{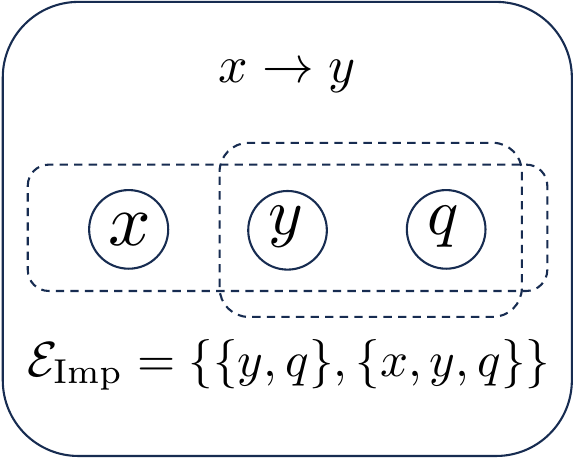}
  \caption{Implication}
  \label{fig:primitive-implication}
\end{subfigure}
\hfill
\begin{subfigure}[b]{0.35\textwidth}
  \centering
  \includegraphics[
    height=3.05cm,
    width=\linewidth,
    keepaspectratio
  ]{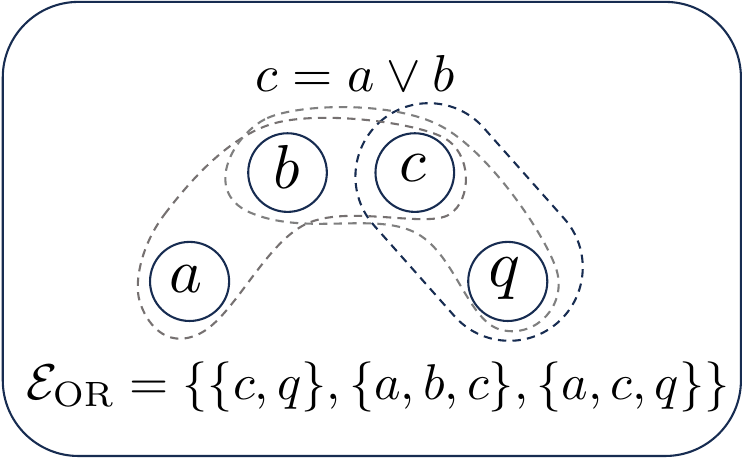}
  \caption{Binary OR}
  \label{fig:primitive-or}
\end{subfigure}

\caption{Local stopping-set constructions.  Each panel shows the
relation on the interface elements together with the hyperedges that
impose it.  The dashed regions represent hyperedges, while the outer
solid border only delimits the panel.  In panels~(b) and~(c), $q$ is
private and is uniquely determined by the interface values.}
\label{fig:local-primitives}
\end{figure}

\begin{lemma}
\label{lem:eq-gadget}
The hyperedge $\{p,q\}$ realizes the equality relation $p=q$.
Equivalently, its feasible assignments are $00$ and $11$.
\end{lemma}

\begin{proof}
The hyperedge $\{p,q\}$ meets the support of the assignment in exactly one
element precisely when one of $p,q$ is one and the other is zero.
Those two assignments are infeasible, while $00$ and $11$ are
feasible.
\end{proof}

\begin{lemma}
\label{lem:imp-gadget}
On the interface elements $(x,y)$,
$\operatorname{Imp}(x,y;q)$ realizes the implication
$x\to y$, equivalently $x\le y$.
More precisely, the family is feasible exactly when $q=y$ and
$x\le y$.
Thus every interface assignment satisfying $x\to y$ has a unique
feasible extension to $q$.
In the coordinate order $(x,y,q)$, the feasible assignments are
$000$, $011$, and $111$.
\end{lemma}

\begin{proof}
The hyperedge $\{y,q\}$ realizes equality and therefore forces $q=y$.
If $y=q=0$, the hyperedge $\{x,y,q\}$ has intersection size $x$ with the
support, so feasibility forces $x=0$.
If $y=q=1$, the same hyperedge has intersection size $2+x$ and is feasible
for either value of $x$.
Hence the interface assignment is feasible exactly when
$x\le y$, which is the Boolean implication $x\to y$, and the private
coordinate is uniquely fixed as $q=y$.
\end{proof}

\begin{lemma}
\label{lem:or-gadget}
On the interface elements $(a,b,c)$,
$\operatorname{OR}(a,b;c,q)$ realizes the relation
$c=a\lor b$.
More precisely, the family is feasible exactly when
$c=q=a\lor b$.
Thus every assignment to the inputs $(a,b)$ has a unique feasible
extension to $(c,q)$.
In the coordinate order $(a,b,c,q)$, the feasible assignments are
$0000$, $0111$, $1011$, and $1111$.
\end{lemma}

\begin{proof}
The hyperedge $\{c,q\}$ realizes equality and therefore forces $q=c$.

Suppose first that $c=q=0$.
Then the hyperedge $\{a,c,q\}$ has intersection size $a$ with the support,
so feasibility forces $a=0$.
With $a=c=0$, the hyperedge $\{a,b,c\}$ has intersection size $b$ and
therefore forces $b=0$.
Thus $c=0$ is feasible exactly when $a=b=0$.

Suppose instead that $c=q=1$.
The hyperedge $\{a,c,q\}$ already contains at least two selected elements
and imposes no restriction on $a$.
The hyperedge $\{a,b,c\}$ has intersection size $1+a+b$ and is feasible
exactly when $a+b\ge1$.
Thus $c=1$ is feasible exactly when at least one of $a,b$ is one.

Combining the two cases gives $c=a\lor b$, and the equality edge
uniquely fixes $q=c$.
\end{proof}

We next use the binary disjunction construction to represent the
disjunction of an arbitrary finite set of inputs.
For a finite set $A$, the intended interface relation is
\(
  o_A=\bigvee_{a\in A}a,
\)
where the empty disjunction is understood to be zero.
Fix an arbitrary ordering $A=\{a_1,\ldots,a_k\}$.
If $k=0$, introduce an output element $o_A$ together with the
singleton hyperedge $\{o_A\}$, which forces $o_A=0$.
If $k=1$, take $o_A=a_1$ and introduce no new element.
For $k\ge2$, introduce intermediate output elements
$o_2,\ldots,o_k$ and private elements $q_2,\ldots,q_k$.
Add
$\operatorname{OR}(a_1,a_2;o_2,q_2)$ and, for each
$3\le i\le k$, add
$\operatorname{OR}(o_{i-1},a_i;o_i,q_i)$.
The output of the resulting OR tree is $o_A=o_k$.

\begin{lemma}
\label{lem:or-tree}
Every assignment to the elements of $A$ has a unique feasible
extension to the elements introduced above, and the output is
$o_A=\bigvee_{a\in A}a$, where the empty disjunction is zero.  Every
auxiliary value is a monotone Boolean function of the input values.
\end{lemma}

\begin{proof}
If $k=0$, the singleton hyperedge $\{o_A\}$ forces $o_A=0$.  The case
$k=1$ is immediate.  Suppose that $k\ge2$.  By
\cref{lem:or-gadget}, the first binary construction uniquely forces
$o_2=a_1\lor a_2$ and $q_2=o_2$.  For every $3\le i\le k$, the next
construction uniquely forces $o_i=o_{i-1}\lor a_i$ and $q_i=o_i$.
Induction gives $o_i=a_1\lor\cdots\lor a_i$ for every
$2\le i\le k$, proving the asserted output value and uniqueness.
Each auxiliary coordinate is the disjunction of a prefix of the input
sequence and is therefore monotone.
\end{proof}

The final local construction introduces a root element $r$ that
separates the all-zero assignment from the represented assignments.
Let $W$ be a finite set, and let $r,\rho\notin W$ be fresh elements.
Add the hyperedge $\{r,\rho\}$ and, for every $w\in W$, the
hyperedge $\{w,r,\rho\}$.

\begin{lemma}
\label{lem:root}
This construction is feasible exactly in the following two cases.
If $r=0$, then $\rho=0$ and every element of $W$ is zero.
If $r=1$, then $\rho=1$, while the values on $W$ are arbitrary.
Consequently, the all-zero assignment is the only feasible assignment
with $r=0$.
\end{lemma}

\begin{proof}
The hyperedge $\{r,\rho\}$ forces $r=\rho$.
If $r=\rho=0$, then each hyperedge $\{w,r,\rho\}$ forces $w=0$.
If $r=\rho=1$, the same hyperedge meets the support in two or three
elements and imposes no condition on $w$.
\end{proof}

\subsection{Encoding Horn models as stopping sets of a hypergraph}
\label{sec:compiler-construction}

We now use the local constructions to encode the models of a Horn CNF
as nonempty stopping sets of a hypergraph.
The encoding uses false-variable coordinates.
More precisely, if $M\subseteq X$ is a source assignment and
$S=X\setminus M$, then a source element $x$ is selected in the target
stopping set exactly when $x$ is false under $M$.

Let $\Gamma$ be a normalized Horn CNF on the variable set $X$.
We write
$\False(\Gamma):=\{X\setminus M:M\models\Gamma\}$
for the family of false-variable sets of its models.
When $S\subseteq X$ is fixed, each source variable $x$ appearing
below denotes the bit $\one[x\in S]$; thus $x=1$ means that $x$ is
false in the original assignment.

Consider a definite clause $A\to b$.
It is violated exactly when $b$ is false and every variable in $A$ is
true.
In false-variable coordinates, it is therefore satisfied exactly when
$b\le\bigvee_{a\in A}a$: whenever $b=1$, at least one premise
variable must also have value one.
A negative clause $A\to\bot$ is satisfied exactly when
$\bigvee_{a\in A}a=1$, since at least one variable in its premise must
be false.
As before, the empty disjunction has value zero.

For each clause, the finite disjunction construction computes the
disjunction of the false-variable bits in its premise.
For a definite clause, we impose the implication from the
false-variable bit of its conclusion to this output.
For a negative clause, we require the same output to equal the root
value.
The root construction is applied to the source elements.
When $r=0$, it forces every source coordinate to zero; the unique
extension properties of the preceding constructions then force every
remaining coordinate to zero.
Thus the empty set is the only stopping set outside the represented
branch.

We now construct the hypergraph $\cC(\Gamma)$.

\begin{enumerate}
  \item For each source variable $x\in X$, introduce a target element,
        also denoted by $x$.

  \item For each clause $C$ with premise $A_C$, apply the finite-input
        disjunction construction to the elements of $A_C$, and denote
        its output by $o_C$.
        All auxiliary elements introduced for different clauses are
        chosen to be distinct.

  \item For each definite clause $C=(A_C\to b_C)$, introduce a private
        element $q_C$ and add
        $\operatorname{Imp}(b_C,o_C;q_C)$.
        This enforces the false-variable condition $b_C\le o_C$.


  \item Introduce fresh elements $r$ and $\rho$ not used elsewhere.
        Add the hyperedge $\{r,\rho\}$ and, for every $x\in X$, the
        hyperedge $\{x,r,\rho\}$.

  \item For each negative clause $C=(A_C\to\bot)$, add the equality
        hyperedge $\{o_C,r\}$.
\end{enumerate}

Let $U_\Gamma$ be the set of all target elements introduced above. 
All auxiliary elements introduced in distinct occurrences of the clause constructions are distinct.

For example, the formula
$\Gamma_{\mathrm{ex}}
=(\{a,b\}\to c)\wedge(\{d\}\to\bot)$
gives the false-variable conditions
$c\le a\lor b$ and $d=1$.
The first is represented by the finite disjunction construction for
$\{a,b\}$ followed by
$\operatorname{Imp}(c,o_C;q_C)$, while the second is represented by
the equality hyperedge $\{d,r\}$.

For every $S\subseteq X$, define the canonical target assignment 
$\eta_\Gamma(S)\subseteq U_\Gamma$ as follows. The selected source 
elements are exactly those in $S$. Every finite disjunction receives 
its unique feasible evaluation from \cref{lem:or-tree}; for every 
definite clause, set $q_C=o_C$; finally, set $r=1$ and 
$\rho=1$. This assignment is defined for every $S\subseteq X$, 
although it need not be a stopping set when 
$X\setminus S$ is not a model of $\Gamma$. 
It can be evaluated in one pass through the construction.


\subsection{The representation theorem}
\label{sec:exact-representation}

We now prove that the construction represents the entire model poset,
not only its maximal elements.

\begin{lemma}
\label{lem:compiler-complete}
If $M\models\Gamma$ and $S=X\setminus M$, then
$\eta_\Gamma(S)$ is a nonempty stopping set of $\cC(\Gamma)$.
\end{lemma}

\begin{proof}
Every finite disjunction is feasible under its canonical evaluation,
by \cref{lem:or-tree}.  Consider a definite clause
$C=(A_C\to b_C)$.  Since $M$ satisfies $C$, its false-variable set
satisfies $b_C\le\bigvee_{a\in A_C}a$.  The disjunction construction
has output $o_C=\bigvee_{a\in A_C}a$, and the canonical target assignment sets
$q_C=o_C$; hence $\operatorname{Imp}(b_C,o_C;q_C)$ is feasible by
\cref{lem:imp-gadget}.

For a negative clause $C=(A_C\to\bot)$, satisfaction by $M$ gives
$o_C=\bigvee_{a\in A_C}a=1$.
Since the canonical target assignment also has $r=1$, the equality hyperedge
$\{o_C,r\}$ is feasible.
The root construction is feasible because $r=\rho=1$.
Thus all hyperedges of $\cC(\Gamma)$ satisfy the stopping-set
condition.
The resulting stopping set is nonempty because it contains $r$.
\end{proof}

\begin{lemma}
\label{lem:compiler-sound}
Let $T$ be a nonempty stopping set of $\cC(\Gamma)$ and put
$S=T\cap X$.  Then $X\setminus S$ is a model of $\Gamma$.
\end{lemma}

\begin{proof}
Suppose first that $r=0$.
By \cref{lem:root}, applied with $W=X$, one has $\rho=0$ and
$T\cap X=\varnothing$.
The unique extension in \cref{lem:or-tree} then forces every element
introduced by a finite disjunction construction to be zero.
For each definite clause, \cref{lem:imp-gadget} further forces
$q_C=o_C=0$.
These are all the elements of $U_\Gamma$, so $T=\varnothing$, contrary
to the assumption.
Hence $r=1$, and the hyperedge $\{r,\rho\}$ gives $\rho=1$.

The hyperedges of each finite disjunction are satisfied by $T$, so
\cref{lem:or-tree} gives
$o_C=\bigvee_{a\in A_C}\mathbf 1[a\in~S]$ for every clause $C$.
For a definite clause $C=(A_C\to b_C)$,
\cref{lem:imp-gadget} yields $b_C\le o_C$, which is exactly the
false-variable condition for that clause.  For a negative clause, the
edge $\{o_C,r\}$ and $r=1$ force $o_C=1$, which is the corresponding
false-variable condition.  Therefore $X\setminus S$ satisfies every
clause of $\Gamma$.
\end{proof}

\begin{lemma}
\label{lem:compiler-unique}
For every $S\subseteq X$, at most one nonempty stopping set $T$ of
$\cC(\Gamma)$ satisfies $T\cap X=S$.  Such a stopping set exists
exactly when $X\setminus S\models\Gamma$, in which case it is
$\eta_\Gamma(S)$.  Moreover, the empty set is the only stopping set
that is not represented by a source model.
\end{lemma}

\begin{proof}
Let $T$ be a nonempty stopping set with $T\cap X=S$.
By the first part of the proof of \cref{lem:compiler-sound},
$r=\rho=1$.
Once the source coordinates are fixed, \cref{lem:or-tree} uniquely
determines every output and private element in each finite
disjunction.
Then \cref{lem:imp-gadget} uniquely determines $q_C=o_C$ for every
definite clause.
These are all the target elements, so $T$ is uniquely determined and
equals $\eta_\Gamma(S)$.

Existence for source models follows from
\cref{lem:compiler-complete}, while \cref{lem:compiler-sound} excludes
an extension of any nonmodel.  The empty set is always a stopping set.
The proof of \cref{lem:compiler-sound} also shows that every stopping
set with $r=0$ is empty.
\end{proof}

\begin{lemma}
\label{lem:compiler-order}
For all $S,S'\in\False(\Gamma)$, one has
$S\subseteq S'$ if and only if
$\eta_\Gamma(S)\subseteq\eta_\Gamma(S')$.  Consequently,
\[
  \eta_\Gamma:\False(\Gamma)
  \longrightarrow
  \Stop(\cC(\Gamma))\setminus\{\varnothing\}
\]
is an inclusion-order isomorphism whose inverse is restriction to
$X$.
\end{lemma}

\begin{proof}
If $\eta_\Gamma(S)\subseteq\eta_\Gamma(S')$, intersecting both sides
with $X$ gives $S\subseteq S'$.  Conversely, suppose that
$S\subseteq S'$.  The source coordinates are then ordered
coordinatewise.  By \cref{lem:or-tree}, every output and every private
coordinate in a finite disjunction is a monotone function of the
source coordinates.  Each $q_C$ equals its disjunction output, while
$r$ and $\rho$ are one in both canonical target assignments.  Hence
$\eta_\Gamma(S)\subseteq\eta_\Gamma(S')$.

Injectivity follows by restriction to $X$, and surjectivity onto the
nonempty stopping sets follows from
\cref{lem:compiler-sound,lem:compiler-unique}.
\end{proof}

\begin{lemma}
\label{lem:compiler-size}
The hypergraph $\cC(\Gamma)$ is normalized, has rank at most three,
and can be constructed in $O(\|\Gamma\|_{\mathrm{inc}})$ time.
Moreover,
$L(\cC(\Gamma))=O(\|\Gamma\|_{\mathrm{inc}})$.
\end{lemma}

\begin{proof}
Let $n:=|X|$.
Let $e$ be the number of clauses with empty premise, let $d$ and $g$
be the numbers of definite and negative clauses, respectively, and
put
$t:=\sum_{C\in\Gamma}\max\{|A_C|-1,0\}$.

The finite disjunction constructions introduce $e+2t$ elements,
$e+3t$ hyperedges, and $e+8t$ incidences.
The definite clauses introduce $d$ elements, $2d$ hyperedges, and
$5d$ incidences.
The negative clauses introduce $g$ hyperedges and $2g$ incidences.
Finally, the root construction introduces two elements,
$n+1$ hyperedges, and $3n+2$ incidences.

Consequently,
\(
  |U_\Gamma|=n+e+2t+d+2,
\)
\(
  |\cC(\Gamma)|=n+e+3t+2d+g+1,
\)
and the number of incidences is
$3n+e+8t+5d+2g+2$.
Thus
\(
  L(\cC(\Gamma))
  =
  5n+3e+13t+8d+3g+5
  =
  O(\|\Gamma\|_{\mathrm{inc}}).
\)

Every hyperedge has size one, two, or three.
All output and private elements belonging to different occurrences
are distinct.
When $|A_C|=1$, the output $o_C$ is the unique source element in
$A_C$, but the clause-specific private element $q_C$ distinguishes
the hyperedges introduced for definite clauses, while normalization
of $\Gamma$ distinguishes the equality hyperedges introduced for
negative clauses.
Together with the normalization of $\Gamma$, this also implies that
the constructed hyperedges are pairwise distinct, so
$\cC(\Gamma)$ is already normalized.
The construction is produced by a single pass through the clauses and
their premises.
\end{proof}

For a source model $M\models\Gamma$, define
\(
  \widehat\eta_\Gamma(M)
  :=
  U_\Gamma\setminus\eta_\Gamma(X\setminus M).
\)
By
\cref{prop:complement-duality,lem:compiler-complete},
this is a proper model of $\Phi_{\cC(\Gamma)}$.
Thus the construction has two equivalent forms.
The map $\eta_\Gamma$ represents the false-variable sets of source
models as nonempty stopping sets, while complementing both the source
and target sets gives the map $\widehat\eta_\Gamma$ between model
families.
These correspondences are summarized by the commutative diagram
\[
\begin{tikzcd}[
  column sep=huge,
  row sep=large
]
\Mod(\Gamma)
  \arrow[r,"X\setminus(\,\cdot\,)"]
  \arrow[d,"\widehat\eta_\Gamma"']
&
\False(\Gamma)
  \arrow[d,"\eta_\Gamma"]
\\
\Mod(\Phi_{\cC(\Gamma)})\setminus\{U_\Gamma\}
  \arrow[r,"U_\Gamma\setminus(\,\cdot\,)"']
&
\Stop(\cC(\Gamma))\setminus\{\varnothing\}.
\end{tikzcd}
\]
The horizontal maps reverse inclusion, whereas the vertical maps
preserve it.
The preceding lemmas show that these maps represent the entire source
model poset, rather than merely preserving satisfiability, and that
all auxiliary coordinates are uniquely determined.
We can now state the main representation theorem.


\begin{theorem}
\label{thm:compiler}
For every normalized Horn CNF $\Gamma$ on $X$, one can construct in
$O(\|\Gamma\|_{\mathrm{inc}})$ time a normalized hypergraph
$\cC(\Gamma)$ on a ground set $U_\Gamma$ such that 
$\rank(\cC(\Gamma))\le3$ and
$L(\cC(\Gamma))=O(\|\Gamma\|_{\mathrm{inc}})$.

The maps
\[
  \eta_\Gamma:
  \False(\Gamma)
  \longrightarrow
  \Stop(\cC(\Gamma))\setminus\{\varnothing\}
\]
and
\[
  \widehat\eta_\Gamma:
  \Mod(\Gamma)
  \longrightarrow
  \Mod(\Phi_{\cC(\Gamma)})\setminus\{U_\Gamma\}
\]
are inclusion-order isomorphisms.
For each source model $M\models\Gamma$, the proper target model 
$\widehat\eta_\Gamma(M)$ is the unique proper model $K$ of $\Phi_{\cC(\Gamma)}$ 
satisfying $K\cap X=M$. Hence every coordinate in $U_\Gamma\setminus X$ is uniquely determined by $M$.
The inverse of $\eta_\Gamma$ is $T\mapsto T\cap X$, and the inverse of
$\widehat\eta_\Gamma$ is $K\mapsto K\cap X$.

In particular, maximal models of $\Gamma$ correspond bijectively to
the coatoms of the closure system
$\Mod(\Phi_{\cC(\Gamma)})$.
For bit-vector representations, the forward and inverse maps can be
evaluated in $O(\|\Gamma\|_{\mathrm{inc}})$ time per output.
\end{theorem}

\begin{proof}
The stopping-set correspondence, uniqueness of the auxiliary
coordinates, and preservation of inclusion follow from
\cref{lem:compiler-complete,lem:compiler-sound,lem:compiler-unique,lem:compiler-order}.  The rank, incidence length,
and construction time follow from \cref{lem:compiler-size}.

Complementation reverses inclusion on both the source and target
sides.  Therefore the map
$M\mapsto U_\Gamma\setminus\eta_\Gamma(X\setminus M)$ is an
inclusion-order isomorphism from $\Mod(\Gamma)$ onto the proper models
of $\Phi_{\cC(\Gamma)}$.  If $K=\widehat\eta_\Gamma(M)$, then
$K\cap X=M$, so restriction to $X$ is its inverse.

An inclusion-maximal source model has an inclusion-minimal
false-variable set.  The first order isomorphism sends such sets
exactly to the inclusion-minimal nonempty stopping sets of
$\cC(\Gamma)$.  By \cref{prop:complement-duality}, their complements
are precisely the coatoms of $\Mod(\Phi_{\cC(\Gamma)})$.  Because
every nonempty stopping set lies in the image of $\eta_\Gamma$, no
additional coatom can arise from auxiliary coordinates.

The forward maps are evaluated by a single pass through the
construction, and the inverse maps are restrictions to $X$.  The
per-output time bound follows from the linear size of
$\cC(\Gamma)$.
\end{proof}

\begin{corollary}
\label{cor:nonempty-premises}
If every clause of $\Gamma$ has a nonempty premise, then every
hyperedge of $\cC(\Gamma)$ has size two or three.
\end{corollary}

\begin{proof}
A unary premise uses its input directly and introduces no hyperedge.
A premise of size at least two uses only the two- and three-element
hyperedges of the binary disjunction construction.
Every definite clause uses one two-element and one three-element
hyperedge, while every negative clause uses one two-element equality
hyperedge.
The root construction uses the two-element hyperedge $\{r,\rho\}$ and
the three-element hyperedges $\{x,r,\rho\}$ for $x\in X$.
The only singleton hyperedge in the general construction is the one
that fixes the output of an empty premise to zero.
\end{proof}

If $\Gamma$ contains the empty negative clause
$\varnothing\to\bot$, the output of its empty disjunction is fixed to
zero and equated with $r$.
A nonempty stopping set would have $r=1$ by \cref{lem:root}, so no such
stopping set exists.
This agrees with the fact that a Horn CNF containing
$\varnothing\to\bot$ has no model.
The construction also applies without change when $\Gamma$ is empty
or when $X=\varnothing$.

From the viewpoint of Boolean constraint languages, the local
constraints use equality, the ternary symmetric relation
$S_{023}=\{000,011,101,110,111\}$, and, for an empty premise, the
unary relation $\{0\}$ represented by a singleton hyperedge.
If the root is included among the visible coordinates, the
construction gives a unique-witness primitive-positive definition of
the augmented relation consisting of the all-zero tuple together with
the tuples $(\chi_S,1)$ for $S\in\False(\Gamma)$.
This is an instance of unique primitive-positive definability
\citep{LagerkvistNordh2019}.
When every clause premise is nonempty, the unary relation is not
needed.

The representation theorem also has the following consequence for
finite closure systems and lattices.

\begin{corollary}
\label{cor:rank3-universality}
Let $\mathcal C$ be a finite closure system on $X$, represented by a
normalized implicational basis $\Gamma$ written in singleton-conclusion form, so that
$\mathcal C=\Mod(\Gamma)$.
One can construct a normalized hypergraph $\cH$ of rank at most three
such that
\(
  \mathcal C
  \cong
  \Mod(\Phi_{\cH})\setminus\{V(\cH)\}
\)
under inclusion.
Moreover,
$L(\cH)=O(\|\Gamma\|_{\mathrm{inc}})$.
Consequently, every finite lattice is isomorphic to the proper-model
poset of a hypergraph Horn function of rank at most three.
\end{corollary}

\begin{proof}
The first statement follows directly from \cref{thm:compiler}.

Let $L$ be a finite lattice and put
$\mathcal C_L:=\{\mathord{\downarrow}x:x\in L\}$.
Since
$\mathord{\downarrow}x\cap\mathord{\downarrow}y
 =\mathord{\downarrow}(x\wedge y)$
and $\mathord{\downarrow}1=L$, the family $\mathcal C_L$ is a closure
system on the ground set $L$.
The map $x\mapsto\mathord{\downarrow}x$ is an inclusion-order
isomorphism from $L$ onto $\mathcal C_L$.
Applying the first statement to an implicational basis of
$\mathcal C_L$ proves the claim.
\end{proof}

The corollary concerns the inclusion order of the model families.
The bound $L(\cH)=O(\|\Gamma\|_{\mathrm{inc}})$ is measured relative
to the incidence size of the supplied implicational basis.
The final statement for an abstract finite lattice is therefore
existential unless such a basis is included as part of its
representation.
The construction introduces auxiliary elements and one additional top
model; in general, $\Gamma$ and $\Phi_{\cH}$ do not represent the same
Boolean function.

When $\Gamma$ is definite, $X$ is the unique maximum model of $\Gamma$. 
By \cref{thm:compiler}, $\widehat{\eta}_\Gamma(X)$ is therefore 
the unique maximum proper model of $\Phi_{\cC(\Gamma)}$, 
and hence its unique coatom. Thus the lattice representation of \cref{cor:rank3-universality} 
does not by itself produce difficult coatom families. 
The hard family used later in \cref{sec:hardness-transfer} contains a
negative clause that excludes the all-true source assignment.
This removes the common maximum, allowing the remaining clauses to
produce the incomparable maximal models used in the hardness
reduction.
The root construction then places the universal target model
$U_\Gamma$ strictly above the represented proper-model poset.

\section{Complexity consequences}
\label{sec:consequences}

This section applies the representation theorem to the computational
problems introduced in \cref{sec:enumeration-prelim}.  We first record
which elementary coatom tasks are polynomial-time solvable and contrast
them with minimum-complement optimization.  We then recall the maximal
Horn model construction of \citet{KavvadiasSideriStavropoulos2000} and
transfer its hardness through the order-preserving correspondence of
\cref{thm:compiler}.

\subsection{Basic computational tasks}
\label{sec:basic-tasks}

We first consider three basic tasks associated with the coatom family.
A proposed coatom can be verified in polynomial time, and one 
arbitrary coatom can be found in polynomial time.
By contrast, deciding whether there is a coatom whose complement has 
size at most a prescribed bound is $\NP$-complete.
We then derive an immediate consequence for enumeration in nondecreasing order of complement size.

\begin{lemma}
\label{lem:coatom-verification}
Let \(\cH\) be a hypergraph on \(V\).  A set \(M\subseteq V\) is a
coatom of the closure system \(\Mod(\Phi_{\cH})\) if and only if
\(\Cl_{\cH}(M)=M\ne V\) and
\(\Cl_{\cH}(M\cup\{v\})=V\) for every \(v\in V\setminus M\).
Consequently, a bit-vector candidate can be verified using at most
\(|V|+1\) closure computations.
\end{lemma}

\begin{proof}
If \(M\) is a coatom, then it is a proper closed set.  For any
\(v\notin M\), the closed set \(\Cl_{\cH}(M\cup\{v\})\) properly
contains \(M\), and hence must be \(V\).

Conversely, suppose that the stated conditions hold.  If a proper
closed set \(D\) satisfied \(M\subsetneq D\subsetneq V\), choose
\(v\in D\setminus M\).  Monotonicity of closure would give
\(\Cl_{\cH}(M\cup\{v\})\subseteq\Cl_{\cH}(D)=D\ne V\), contrary to
the second condition.  Thus no proper closed set properly contains
\(M\), so \(M\) is a coatom.
\end{proof}

\begin{proposition}
\label{prop:find-one}
Given a hypergraph \(\cH\) on \(V\), one can find a coatom of
\(\Mod(\Phi_{\cH})\), or determine that none exists, using at most
\(|V|+1\) closure computations.  For a normalized input, the running
time is \(O(|V|L(\cH))\), and the working space is \(O(L(\cH))\).
\end{proposition}

\begin{proof}
Compute \(C=\Cl_{\cH}(\varnothing)\).  If \(C=V\), then every closed
set contains \(C\), so \(V\) is the only closed set and no coatom
exists.

Assume that \(C\ne V\), fix an order of the vertices, and process each
vertex once.  When the current vertex \(v\) is not already in \(C\),
compute \(D=\Cl_{\cH}(C\cup\{v\})\).  If \(D\ne V\), replace \(C\)
by \(D\); otherwise reject \(v\) and leave \(C\) unchanged.  A vertex
rejected at a closed set \(C_0\) can never become admissible later:
for every later current set \(C\supseteq C_0\), monotonicity gives
\(V=\Cl_{\cH}(C_0\cup\{v\})\subseteq
\Cl_{\cH}(C\cup\{v\})\), and therefore the latter closure is also
\(V\).

At termination, every vertex outside \(C\) has been rejected at some
closed set contained in the final \(C\).  Hence
\(\Cl_{\cH}(C\cup\{v\})=V\) for every \(v\notin C\).  The set \(C\)
is proper throughout the algorithm, so \cref{lem:coatom-verification}
shows that the final \(C\) is a coatom.  There is one initial closure
computation and at most one further computation for each vertex,
which gives the stated bounds.
\end{proof}

\begin{proposition}
\label{prop:min-complement-hard}
Given a hypergraph \(\cH\) on \(V\) and an integer \(k\), deciding
whether \(\Mod(\Phi_{\cH})\) has a coatom \(M\) with
\(|V\setminus M|\le k\) is \(\NP\)-complete.  Consequently, finding a
minimum-cardinality nonempty stopping set, or equivalently a
maximum-cardinality coatom, is \(\NP\)-hard.
\end{proposition}

\begin{proof}
The problem belongs to \(\NP\) by
\cref{lem:coatom-verification}.

For hardness, let $H$ be a binary parity-check matrix.
Duplicate rows may be removed and zero rows may be discarded without
changing the stopping-set family.
Take the columns of the resulting matrix as the ground set and the
distinct nonempty row supports as the hyperedges.
A subset of columns is then a stopping set of the Tanner graph exactly
when no row support meets it in one coordinate. Krishnan and Shankar
proved that deciding whether such a Tanner graph has a nonempty
stopping set of size at most \(k\) is \(\NP\)-complete
\citep{KrishnanShankar2007}.

Every nonempty stopping set contains an inclusion-minimal nonempty
stopping set of no larger cardinality: repeatedly delete an element
whenever the stopping-set condition remains satisfied.  Thus a
nonempty stopping set of size at most \(k\) exists if and only if a
minimal nonempty stopping set of size at most \(k\) exists.  By
\cref{prop:complement-duality}, the latter sets are precisely the
complements of coatoms, which proves the reduction.
\end{proof}

\begin{corollary}
\label{cor:size-ordered-enum}
Unless $\mathsf P=\mathsf{NP}$, there is no polynomial-delay 
algorithm that enumerates the coatoms of 
an arbitrary hypergraph in nondecreasing order of complement size.
\end{corollary}

\begin{proof}
Suppose that such an algorithm exists. 
Given $\cH$ and $k$, run it until its first output or termination. 
If it terminates without an output, then no coatom exists. 
Otherwise, let $M$ be its first output. 
Since the outputs are ordered by nondecreasing complement size, 
$|V\setminus M|$ is minimum among all coatoms. 
Hence a coatom with complement size at most $k$ exists 
if and only if $|V\setminus M|\le k$. 
This would solve the problem of \cref{prop:min-complement-hard} in polynomial time.
\end{proof}

\subsection{Hardness transfer from maximal Horn models}
\label{sec:hardness-transfer}

We now transfer hardness results for maximal Horn models through the
representation theorem.
The source is the construction of
\citet{KavvadiasSideriStavropoulos2000}.
By \cref{thm:compiler}, maximal source models correspond bijectively
to target coatoms, the original coordinates are preserved, and the
forward and inverse maps are computable in polynomial time.
These properties allow us to transfer partial output families,
coordinate constraints, and complete enumeration.

For a Horn CNF $\Gamma$, let $\MaxMod(\Gamma)$ denote the family of
its inclusion-maximal models; thus
$\MaxMod(\Gamma):=\max_{\subseteq}\Mod(\Gamma)$.
The incremental maximal-model problem asks, given $\Gamma$ and an
explicit family $\mathcal L\subseteq\MaxMod(\Gamma)$, whether
$\MaxMod(\Gamma)\setminus\mathcal L$ is nonempty.

The following statement combines
\citet[Theorem~4 and Corollary~5]
{KavvadiasSideriStavropoulos2000} with the properties of the
reduction used in their proof.

\begin{theorem}[Kavvadias--Sideri--Stavropoulos]
\label{thm:kss}
The incremental maximal-model problem for Horn CNFs is
$\NP$-complete, and maximal models of Horn CNFs cannot be enumerated
in output-polynomial time unless $\mathsf P=\mathsf{NP}$.

More specifically, from an instance $F$ of
\textsc{Positive One-in-Three 3SAT} with $m$ clauses, the reduction
constructs in polynomial time a Horn CNF $\Gamma_F$, a distinguished
variable $z$, and a family $\mathcal L_F$ of exactly $m$ maximal
models such that $\MaxMod(\Gamma_F)=\mathcal L_F$ when $F$ is a
no-instance, whereas $F$ is a yes-instance exactly when
$\Gamma_F$ has a maximal model omitting $z$.
The formula $\Gamma_F$ and the complete encoding of
$\mathcal L_F$ have polynomial size.
\end{theorem}

The same reduction gives the restricted-family lower bound used for
the enumeration transfer.

\begin{lemma}
\label{lem:kss-family-hard}
Unless $\mathsf P=\mathsf{NP}$, the maximal models of the formulas
$\Gamma_F$ cannot be enumerated in output-polynomial time.
\end{lemma}

\begin{proof}
Suppose that such an enumerator exists, and let $p$ be a polynomial
bounding its running time in the combined input and output length.
Given $F$, construct $\Gamma_F$ and $\mathcal L_F$ as in
\cref{thm:kss}, and let $B_F$ be their combined encoded length.
Run the enumerator on $\Gamma_F$ for $p(B_F)$ steps.

If it outputs a maximal model outside $\mathcal L_F$, then $F$ is a
yes-instance.
If it terminates after outputting exactly $\mathcal L_F$, then $F$ is
a no-instance.
If it has not terminated within $p(B_F)$ steps, then $F$ is a
yes-instance: if $F$ were a no-instance, the complete output would be
$\mathcal L_F$, and the assumed running-time bound would force the
enumerator to terminate within $p(B_F)$ steps.
Thus the enumerator would decide
\textsc{Positive One-in-Three 3SAT} in polynomial time.
\end{proof}

Appendix~\ref{app:kss} gives the construction in the notation 
used here and verifies these properties directly.
Before applying the compiler, we normalize $\Gamma_F$ by removing
duplicate clauses; this does not change its model family, and all
clause premises remain nonempty.
By \cref{cor:nonempty-premises}, every hyperedge of
$\cC(\Gamma_F)$ therefore has size two or three.

We first transfer the incremental hardness result.

\begin{theorem}
\label{thm:incremental-rank3}
The problem \IncHH{} is $\NP$-complete even when every input 
hyperedge has size two or three. 
The hardness also holds for the promise version in which 
every member of the input family is guaranteed to be a coatom.
\end{theorem}

\begin{proof}
For membership in $\NP$, use a coatom not contained in the given
family as a certificate and verify it, together with every listed
set, using \cref{lem:coatom-verification}.

For hardness, let $F$ be an instance of
\textsc{Positive One-in-Three 3SAT}, and construct
$\Gamma_F$ and $\mathcal L_F$ as in \cref{thm:kss}.
Construct $\cC(\Gamma_F)$ and replace each
$M\in\mathcal L_F$ by the coatom
$\widehat\eta_{\Gamma_F}(M)$.
By \cref{thm:compiler}, this map is a bijection from
$\MaxMod(\Gamma_F)$ onto the coatom family of
$\Phi_{\cC(\Gamma_F)}$.
Therefore the mapped family omits a target coatom if and only if
$\mathcal L_F$ omits a maximal model of $\Gamma_F$.
By \cref{thm:kss}, this occurs exactly when $F$ is a yes-instance.
The target hypergraph and the mapped family have polynomial size.
Since every premise of $\Gamma_F$ is nonempty,
\cref{cor:nonempty-premises} shows that every target hyperedge has
size two or three.

The mapped family consists entirely of coatoms, 
so the same reduction proves hardness for the promise version.
\end{proof}

The distinguished variable $z$ also transfers the hardness to the
constrained extension problem.

\begin{theorem}
\label{thm:extension-23}
The problem \HHEx{} is $\NP$-complete even when every input
hyperedge has size two or three.
The hardness holds with $P=\varnothing$ and $|N|=1$.
\end{theorem}

\begin{proof}
Membership in $\NP$ follows from
\cref{lem:coatom-verification}.

Let $F$ be an instance of
\textsc{Positive One-in-Three 3SAT}, and construct
$\Gamma_F$ and its distinguished variable $z$ as in
\cref{thm:kss}.
Apply \cref{thm:compiler} to obtain $\cC(\Gamma_F)$.
Let $z$ also denote the corresponding target element, and set
$P=\varnothing$ and $N=\{z\}$.

For every maximal model $M$ of $\Gamma_F$, the corresponding target
coatom $K=\widehat\eta_{\Gamma_F}(M)$ satisfies $K\cap X=M$.
Consequently, $K\cap N=\varnothing$ if and only if $z\notin M$.
By \cref{thm:kss}, such a maximal model exists exactly when $F$ is a
yes-instance.

Since every premise of $\Gamma_F$ is nonempty,
\cref{cor:nonempty-premises} shows that every target hyperedge has
size two or three.
\end{proof}

Finally, the same correspondence transfers the lower bound for
complete enumeration.

\begin{theorem}
\label{thm:rank3-hardness}
Unless $\mathsf P=\mathsf{NP}$, \HHEnum{} has no output-polynomial
algorithm, even when every input hyperedge has size two or three.
\end{theorem}

\begin{proof}
Suppose that such an algorithm exists.
Given an instance $F$ of
\textsc{Positive One-in-Three 3SAT}, construct $\Gamma_F$ as in
\cref{thm:kss}, and then construct $\cC(\Gamma_F)$ using
\cref{thm:compiler}.
Run the assumed coatom enumerator on $\cC(\Gamma_F)$ and replace each
output coatom $K$ by its restriction $K\cap X$.
By \cref{thm:compiler}, this produces every maximal model of
$\Gamma_F$ exactly once.

The incidence length of $\cC(\Gamma_F)$ is
$O(\|\Gamma_F\|_{\mathrm{inc}})$, and every target coatom is encoded
by a bit vector of length
$O(\|\Gamma_F\|_{\mathrm{inc}})$.
If there are $N$ outputs, their total target encoding length is
therefore $O(N\|\Gamma_F\|_{\mathrm{inc}})$.
Since source models are also represented by bit vectors and
$\Gamma_F$ has at least one variable, $N$ is at most the total
encoded length of the source output.
Consequently, the resulting maximal-model enumerator runs in time
polynomial in the source input length and the total source output
length.

This would give an output-polynomial algorithm for enumerating the
maximal models of the formulas $\Gamma_F$, contradicting
\cref{lem:kss-family-hard}.
Finally, every premise of $\Gamma_F$ is nonempty, so
\cref{cor:nonempty-premises} ensures that every hyperedge of
$\cC(\Gamma_F)$ has size two or three.
\end{proof}

The results proved so far are summarized in \cref{tab:tasks}.

\begin{table}[t]
\centering
\caption{Complexity of the coatom problems considered in this section.}
\label{tab:tasks}
\small
\begin{tabularx}{\textwidth}{
  @{}
  >{\raggedright\arraybackslash}p{0.40\textwidth}
  >{\raggedright\arraybackslash}X
  @{}
}
\toprule
Problem or task
  & Result \\
\midrule
Coatom verification
  & solvable in polynomial time for arbitrary hypergraphs \\

Finding one coatom, or determining that none exists
  & solvable in polynomial time for arbitrary hypergraphs \\

Existence of a coatom $M$ with $|V\setminus M|\le k$
  & $\NP$-complete for general hypergraphs \\
  
Enumeration in nondecreasing order of complement size
  & not in $\DelayP$ unless $\mathsf P=\mathsf{NP}$, for general hypergraphs \\

\IncHH{}
  & $\NP$-complete even when every hyperedge has size two or three \\

\HHEx{}
  & $\NP$-complete even when every hyperedge has size two or three \\

\HHEnum{}
  & not in $\OutputP$ unless $\mathsf P=\mathsf{NP}$, even when
    every hyperedge has size two or three \\
\bottomrule
\end{tabularx}
\end{table}

Thus the enumeration lower bound is not a first-solution lower bound: 
an arbitrary coatom can be found in polynomial time. 
Hardness arises when one imposes a size objective, 
asks for a coatom satisfying additional constraints, 
asks whether a valid partial output family omits a coatom, or generates the entire family.

\section{Rank and frequency thresholds}
\label{sec:thresholds}

The lower bounds of \cref{sec:consequences} already hold when every
hyperedge has size two or three, but they do not initially bound
element frequency.
This section identifies the adjacent structural thresholds.
Rank at most two and frequency at most two both admit output-linear 
total-time coatom generation after normalization. 
The frequency-two case also admits a polynomial-delay, 
polynomial-space algorithm through its graphic-circuit representation.

An incidence-splitting transformation preserves the stopping-set
order and reduces maximum frequency to three.
Consequently, unless $\mathsf P=\mathsf{NP}$, coatom enumeration is
not in $\OutputP$ even when every hyperedge has size two or three and
every element has frequency at most three.
Extension has a different boundary: it is polynomial for rank at most two, 
but remains $\NP$-complete at exact frequency two. The final subsection 
replaces the remaining two-element hyperedges and extends the enumeration 
lower bound to three-uniform hypergraphs without increasing maximum element frequency. 
Under the parity-check interpretation, hypergraph rank is the maximum check-node 
degree and element frequency is the maximum variable-node degree, 
so these results also identify degree thresholds for Tanner graphs.

\subsection{Rank at most two}
\label{sec:rank2}

The rank-two case is elementary, but it gives the positive endpoint
needed for the rank classification.
A hypergraph of rank at most two is an ordinary graph together with
possible singleton hyperedges.
The two-element hyperedges identify vertices in the same connected
component, while a singleton hyperedge forces its entire component
into every closed set.
When no singleton hyperedge is present, this is the
equivalence-relation case discussed by
\citet{BercziBorosMakino2024}.
We record the resulting structure because it also gives direct
algorithms for extension and enumeration.

Assume that $\rank(\cH)\le2$.
Let $G_{\cH}$ be the graph on $V$ whose edges are the two-element
hyperedges of $\cH$, and let $D_1,\ldots,D_t$ be its connected
components, including isolated vertices.
Call a component \emph{forced} if it contains a vertex $v$ such that
$\{v\}\in\cH$.
Let $F$ be the union of the forced components, and denote the
remaining components by $C_1,\ldots,C_q$.

\begin{proposition}
\label{prop:rank2-structure}
For every $Y\subseteq V$,
\[
  \Cl_{\cH}(Y)
  =
  F\cup
  \bigcup_{D_i\cap Y\ne\varnothing}D_i.
\]
Consequently, the closed sets are precisely the sets
$F\cup\bigcup_{i\in I}C_i$, where $I\subseteq[q]$, and the stopping
sets are precisely the unions of unforced components.

If $q=0$, then $V$ is the only closed set and
$\Phi_{\cH}$ has no coatom.
If $q>0$, then
\[
  \cM(\Phi_{\cH})
  =
  \{V\setminus C_i:i\in[q]\},
  \qquad
  \MinStop(\cH)
  =
  \{C_i:i\in[q]\}.
\]
\end{proposition}

\begin{proof}
A two-element hyperedge $\{u,v\}$ contributes the implications
$u\to v$ and $v\to u$.
Hence a closed set containing one vertex of a component of
$G_{\cH}$ must contain the entire component.
A singleton hyperedge $\{v\}$ contributes the implication
$\varnothing\to v$.
It therefore forces $v$, and hence the whole component containing
$v$, into every closed set.
Since no implication connects distinct components, the closure of
$Y$ is obtained by taking all forced components together with every
component that meets $Y$.
This proves the closure formula and the description of the closed
sets.

In stopping-set coordinates, a two-element hyperedge forces its two
endpoint bits to agree.
A singleton hyperedge forces its unique bit to be zero, and equality
then forces every bit in the same component to be zero.
Thus a stopping set is an arbitrary union of the unforced components.
Its inclusion-minimal nonempty members are exactly
$C_1,\ldots,C_q$.
The coatom description follows from
\cref{prop:complement-duality}.
\end{proof}

The same component description gives a direct criterion for
constrained extension.

\begin{corollary}
\label{cor:rank2-extension}
Let $P,N\subseteq V$ be disjoint.
There exists a coatom $M$ satisfying $P\subseteq M$ and
$M\cap N=\varnothing$ if and only if some unforced component $C_i$
contains $N$ and is disjoint from $P$.
After the component identifiers and forced status have been stored,
this condition can be tested in $O(|P|+|N|+1)$ time.
\end{corollary}

\begin{proof}
By \cref{prop:rank2-structure}, every coatom has the form
$V\setminus C_i$ for an unforced component $C_i$.
For such a coatom, the condition $P\subseteq V\setminus C_i$ is
equivalent to $P\cap C_i=\varnothing$, while
$(V\setminus C_i)\cap N=\varnothing$ is equivalent to
$N\subseteq C_i$.

To test these conditions, mark the unforced component identifiers met
by $P$.
If $N\ne\varnothing$, reject unless all elements of $N$ lie in one
common unforced component; that component is admissible exactly when
its identifier is unmarked.
If $N=\varnothing$, an admissible coatom exists exactly when some
unforced component remains unmarked.
\end{proof}

The unforced components also give an output-linear enumeration after
normalization.

\begin{theorem}
\label{thm:rank2-enum}
The connected components of $G_{\cH}$ and their forced status can be
computed in $O(L(\cH))$ time using $O(L(\cH))$ working space.
All coatoms can then be enumerated exactly once as follows.

If a coatom is output through its sparse complement $C_i$, the delay
is $\Theta(|C_i|+1)$ and the total running time is
\(
  O\left(
    L(\cH)+\sum_{i=1}^{q}(|C_i|+1)
  \right).
\)
If coatoms are output as $|V|$-bit characteristic vectors, the delay
is $\Theta(|V|)$ and the total running time is
$O(L(\cH)+q|V|)$.
The same bounds hold when outputs are ordered by nondecreasing
complement size.
\end{theorem}

\begin{proof}
Construct adjacency lists for the two-element hyperedges and compute
the components of $G_{\cH}$ by DFS or BFS.
The singleton hyperedges identify the forced components, and each
unforced component is stored as a list of its vertices.
By \cref{prop:rank2-structure}, these lists are exactly the minimal
nonempty stopping sets, and their complements are exactly the
coatoms.

Writing the sparse complement $C_i$ requires
$\Theta(|C_i|+1)$ time, whereas writing the characteristic vector of
$V\setminus C_i$ requires $\Theta(|V|)$ time.
Since the unforced components are pairwise disjoint,
$\sum_{i=1}^{q}|C_i|\le |V|$, which gives the stated total-time
bounds.

Finally, the component sizes are integers between $1$ and $|V|$.
They can therefore be stably bucketed by size in linear additional
time and space, without changing the preprocessing, delay, or
total-time bounds.
The case $q=0$ is detected during preprocessing.
\end{proof}

\subsection{Reducing maximum element frequency to three}
\label{sec:splitting}

The representation theorem controls hyperedge size but not element frequency: 
if $\Gamma$ has $n$ source variables and $g$ negative clauses, 
then $\rho$ and $r$ occur in $n+1$ and $n+g+1$ hyperedges 
of $\cC(\Gamma)$, respectively. 
We therefore replace each incidence of an element by a separate occurrence 
and link all occurrences of the same source element by two-element 
equality hyperedges. The resulting transformation preserves 
the entire stopping-set poset while reducing maximum element frequency to three.

For each incidence $(v,A)$ with $A\in\cH$ and $v\in A$, introduce a
distinct occurrence element $v_A$.
For every hyperedge $A\in\cH$, replace $A$ by the hyperedge
$A^\circ:=\{v_A:v\in A\}$.
Fix a source element $v\in V$, and let
$\mathcal O_v:=\{v_A:A\in\cH,\ v\in A\}$ be the set of its occurrence
elements.
Thus $d_v:=|\mathcal O_v|$ is the frequency of $v$ in $\cH$.
Introduce one additional element $t_v$, called the terminal of $v$.
Figure~\ref{fig:incidence-splitting} illustrates the construction when
$v$ belongs to four hyperedges $A_1,\ldots,A_4$.

To force all elements representing $v$ to have the same stopping-set
membership bit while keeping their frequencies bounded, connect
$\mathcal O_v\cup\{t_v\}$ by a tree $T_v$ of maximum degree at most
three.
If $d_v=0$, let $T_v$ consist only of the isolated terminal $t_v$.
If $d_v=1$, join the unique occurrence element directly to $t_v$.
If $d_v\ge2$, introduce fresh internal vertices and choose a tree of maximum
degree at most three whose leaves are exactly the elements of
$\mathcal O_v\cup\{t_v\}$.
Such a tree can be constructed with $O(d_v)$ vertices and edges.
The vertices denoted by $u_1,u_2,u_3$ in
\cref{fig:incidence-splitting} are such internal vertices.
Every edge of $T_v$ is inserted as a two-element hyperedge.
Let $B_v:=V(T_v)$ be the block representing the source element $v$.
The internal vertices, terminals, and occurrence elements used for
different source elements are all distinct, so the blocks $B_v$ are
pairwise disjoint.

We define $\operatorname{Split}(\cH)$ to be the hypergraph whose
ground set is $\bigcup_{v\in V}B_v$ and whose hyperedges are the
replacement hyperedges $A^\circ$ for $A\in\cH$, together with the
two-element sets corresponding to all edges of the trees $T_v$.
Thus every source hyperedge $A$ is replaced by $A^\circ$, and for
each $v\in A$ this replacement contains exactly one element of
$B_v$, namely $v_A$.

\begin{figure}[t]
\centering
\includegraphics[width=0.50\textwidth]{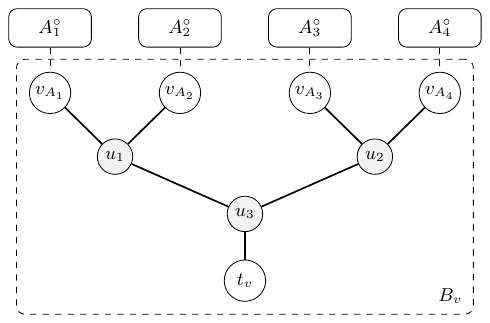}
\caption{The block $B_v$ for a source element $v$ belonging to four
hyperedges $A_1,\ldots,A_4$.
Each incidence $v\in A_i$ is represented by a distinct occurrence
element $v_{A_i}$ contained in the replacement hyperedge
$A_i^\circ$.
The dashed attachments indicate this membership and are not
hyperedges.
The vertices $u_1,u_2,u_3$ are fresh internal vertices of the
subcubic tree $T_v$, while every solid edge is a two-element
hyperedge.
These edges force every stopping set to be constant on $B_v$.
The terminal $t_v$ is later used to transfer extension constraints.}
\label{fig:incidence-splitting}
\end{figure}

For $S\subseteq V$, let
\(
  \widehat S:=\bigcup_{v\in S}B_v.
\)
The equality hyperedges in each tree $T_v$ force every stopping set
of $\operatorname{Split}(\cH)$ to be constant on the block $B_v$.
Moreover, a replacement hyperedge $A^\circ$ contains exactly one
element from $B_v$ for each $v\in A$.
It therefore meets $\widehat S$ in the same number of elements as
$A$ meets $S$.
These observations give an exact correspondence between the stopping
sets of the two hypergraphs.

\begin{theorem}
\label{thm:splitting}
The map $S\mapsto\widehat S$ is an inclusion-order isomorphism from
$\Stop(\cH)$ onto
$\Stop(\operatorname{Split}(\cH))$.
Its inverse sends a stopping set $T$ to
$\{v\in V:t_v\in T\}$.
The map preserves the empty set and hence restricts to a bijection
between the minimal nonempty stopping sets.
Under complement duality, the induced bijection between coatom
families sends a source coatom $M$ to
\(
  V(\operatorname{Split}(\cH))
  \setminus
  \widehat{\,V\setminus M\,}.
\)

Furthermore,
$\Delta(\operatorname{Split}(\cH))\le3$ and
$\rank(\operatorname{Split}(\cH))
\le\max\{2,\rank(\cH)\}$.
The split hypergraph has incidence length $O(L(\cH))$ and can be
constructed in $O(L(\cH))$ time.
The forward and inverse set maps are computable in time linear in
their input and output encodings.
If $\cH$ is normalized, then
$\operatorname{Split}(\cH)$ is normalized.
If every hyperedge of $\cH$ has size two or three, then the same is
true of every hyperedge of $\operatorname{Split}(\cH)$.
\end{theorem}

\begin{proof}
Let $T$ be a stopping set of
$\operatorname{Split}(\cH)$.
For an arbitrary vertex $v\in V$, every edge of $T_v$ is a
two-element hyperedge and therefore forces its two endpoint membership
bits in $T$ to agree.
Since $T_v$ is connected, either all elements of $B_v$ belong to $T$
or none of them do.
This is also immediate when $d_v=0$, because then
$B_v=\{t_v\}$.
Thus $T$ is constant on every block $B_v$.
Since the blocks are pairwise disjoint and nonempty, there is a unique
set $S\subseteq V$ such that $T=\widehat S$, namely
$S=\{v\in V:t_v\in T\}$.

For every source hyperedge $A\in\cH$, the replacement hyperedge
$A^\circ$ contains the occurrence element $v_A$ from $B_v$ for each
$v\in A$, and no other element from that block.
Consequently,
\(
  |A^\circ\cap\widehat S|=|A\cap S|.
\)
Thus $A^\circ$ satisfies the stopping-set condition with respect to
$\widehat S$ exactly when $A$ satisfies it with respect to $S$.
Every tree edge also satisfies the condition, because either both of
its endpoints belong to $\widehat S$ or neither does.
It follows that
\(
  S\in\Stop(\cH)
  \Leftrightarrow
  \widehat S\in
  \Stop(\operatorname{Split}(\cH)).
\)

The map plainly preserves inclusion.
It also reflects inclusion: if
$\widehat S\subseteq\widehat{S'}$ and $v\in S$, then
$t_v\in\widehat S\subseteq\widehat{S'}$, so $v\in S'$.
Since every block is nonempty,
$\widehat S=\varnothing$ holds exactly when $S=\varnothing$.
The map therefore preserves inclusion-minimal nonempty stopping sets.
The stated coatom correspondence follows by applying
\cref{prop:complement-duality} to the source and target hypergraphs.

An occurrence element $v_A$ lies in the replacement hyperedge
$A^\circ$ and in one tree edge, so its frequency is two.
A terminal belongs to at most one tree edge, while an internal vertex
of $T_v$ belongs to at most three tree edges.
Hence
$\Delta(\operatorname{Split}(\cH))\le3$.
The replacement hyperedges have the same sizes as the original
hyperedges, and every additional hyperedge has size two.
This proves the rank bound and the preservation of edge sizes two
and three.
Distinct source hyperedges yield distinct replacement hyperedges
because the occurrence elements are incidence-specific.
The tree hyperedges are distinct and lie within single blocks, whereas
every replacement hyperedge is either a singleton or meets at least
two blocks.
Hence no tree hyperedge coincides with a replacement hyperedge, and
$\operatorname{Split}(\cH)$ is normalized.

Finally, put
$I=\sum_{A\in\cH}|A|$.
For each $v$, the tree $T_v$ has
$O(d_v+1)$ vertices and edges, and
$\sum_{v\in V}d_v=I$.
All trees together therefore have
$O(|V|+I)$ vertices, edges, and incidences.
The replacement hyperedges contribute $|\cH|$ hyperedges and exactly
$I$ incidences.
Thus
$L(\operatorname{Split}(\cH))=O(L(\cH))$, and the construction can
be carried out in linear time.
The forward map outputs the blocks indexed by $S$, while the inverse
map reads the terminal of each block, giving the claimed bounds for
the set maps.
\end{proof}

The terminals allow extension constraints to be transferred without
expanding the prescribed sets.
For disjoint $P,N\subseteq V$, put
$P^\star:=\{t_v:v\in P\}$ and
$N^\star:=\{t_v:v\in N\}$.
If $M=V\setminus S$ is a source coatom and
$M^\star:=V(\operatorname{Split}(\cH))\setminus\widehat S$ is the
corresponding target coatom, then
$P\subseteq M$ and $M\cap N=\varnothing$ hold exactly when
$P^\star\subseteq M^\star$ and
$M^\star\cap N^\star=\varnothing$.
Combining the splitting transformation with the lower bounds of
\cref{sec:consequences} yields simultaneous rank and frequency
bounds.

\begin{corollary}
\label{cor:rank3-freq3}
\label{cor:extension-rank3-freq3}
\label{cor:incremental-rank3-freq3}
Even when every hyperedge has size two or three and every ground 
element has frequency at most three, \IncHH{} and \HHEx{} are 
$\NP$-complete, and \HHEnum{} is not in $\OutputP$ unless 
$\mathsf P=\mathsf{NP}$. 
The incremental hardness also holds for the promise version, 
and the extension hardness still holds with $P=\varnothing$ 
and a singleton forbidden set.
\end{corollary}

\begin{proof}
For \IncHH{}, compose the reduction of
\cref{thm:incremental-rank3} with
\cref{thm:splitting} and map the given list through the induced
coatom bijection.
The source family consists entirely of coatoms, 
and the splitting map is bijective on coatoms, so the promise is preserved.

For \HHEx{}, start from \cref{thm:extension-23}.
Its prescribed sets satisfy $P=\varnothing$ and $N=\{z\}$; after
splitting, use $P^\star=\varnothing$ and
$N^\star=\{t_z\}$.
The terminal equivalence above preserves the extension condition.

For \HHEnum{}, apply the splitting transformation to the hard
instances of \cref{thm:rank3-hardness}.
The transformation is linear in size and bijective on outputs, with
linear-time forward and inverse maps, so an output-polynomial
enumerator for the split instances would give one for the original
instances.
The structural bounds follow from \cref{thm:splitting}.
\end{proof}

\subsection{Frequency at most two}
\label{sec:freq2}

The frequency-two case admits a direct graph-theoretic
representation.
When every ground element belongs to at most two hyperedges, the
hyperedges can be viewed as vertices of a multigraph and the ground
elements as labeled edges.
Under this representation, the minimal nonempty stopping sets are 
exactly the graphic circuits. 
This yields output-linear total-time coatom enumeration and, 
separately, a polynomial-delay, polynomial-space algorithm, 
although constrained extension remains $\NP$-complete 
on the three-uniform subclass of exact frequency two.

Assume that $\Delta(\cH)\le2$.
Introduce a new vertex $\alpha$, called the \emph{apex}, and construct
a labeled multigraph $Q_{\cH}$ with vertex set
$\cH\cup\{\alpha\}$.
The apex supplies a missing incidence endpoint for elements that
belong to fewer than two hyperedges.
For each $x\in V$, introduce an edge $e_x$ as follows.
If $x$ belongs to two hyperedges $A$ and $B$, then $e_x$ joins
$A$ and $B$.
If $x$ belongs to exactly one hyperedge $A$, then $e_x$ joins $A$ to
$\alpha$.
If $x$ belongs to no hyperedge, then $e_x$ is a loop at $\alpha$.
Parallel edges are retained.
The three cases are shown in
\cref{fig:apex-multigraph}.

\begin{figure}[t]
\centering
\includegraphics[width=0.88\textwidth]{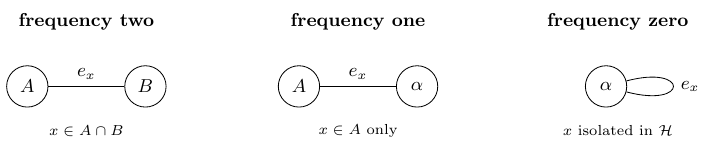}
\caption{The edge of $Q_{\cH}$ associated with a ground element of
frequency two, one, or zero.
Hyperedges of $\cH$ become vertices, while the apex $\alpha$ supplies
any missing incidence endpoint.}
\label{fig:apex-multigraph}
\end{figure}

For $S\subseteq V$, let $E_S:=\{e_x:x\in S\}$.
For every hyperedge vertex $A\in\cH$, an edge $e_x$ is incident with
$A$ exactly when $x\in A$.
Hence
$\deg_{E_S}(A)=|A\cap S|$.
It follows that $S$ is a stopping set of $\cH$ exactly when no
hyperedge vertex has degree one in the subgraph with edge set $E_S$.
No degree condition is imposed at the apex.

\begin{proposition}
\label{prop:freq2-structure}
Suppose that $\Delta(\cH)\le2$.
A nonempty set $S\subseteq V$ is an inclusion-minimal stopping set if
and only if $E_S$ is a circuit of the graphic matroid of
$Q_{\cH}$.
Equivalently, $E_S$ is a loop at the apex, a pair of parallel edges,
or the edge set of a simple cycle of length at least three.
\end{proposition}

\begin{proof}
Let $F=E_S$ be inclusion-minimal among the nonempty edge sets for
which no hyperedge vertex has degree one.
If the subgraph with edge set $F$ had more than one edge-containing
connected component, the edge set of any one component would still
give degree zero or at least two at every hyperedge vertex.
It would therefore correspond to a smaller nonempty stopping set,
contradicting minimality.
Thus the subgraph with edge set $F$ is connected.

Suppose that it contains no graphic circuit.
Then it has no loop or pair of parallel edges and, being connected and
acyclic, is a nontrivial tree.
If the tree does not contain the apex, every leaf is a hyperedge
vertex.
If it contains the apex, it has at least two leaves and at most one
of them is the apex.
In either case some hyperedge vertex has degree one, a contradiction.
Hence $F$ contains a graphic circuit.

The edge set of a graphic circuit satisfies the required degree
condition.
An apex loop has no non-apex endpoint, while every non-apex vertex of
a parallel pair or a simple cycle has degree two.
Since $F$ is inclusion-minimal, it must equal the circuit that it
contains.

Conversely, an apex loop has no nonempty proper subset.
A nonempty proper subset of a parallel pair consists of one edge and
gives degree one at a non-apex endpoint.
A nonempty proper subset of a simple cycle is a disjoint union of
paths, at least one of which has a non-apex endpoint.
That endpoint has degree one.
Thus no nonempty proper subset of a graphic circuit is a stopping set,
so every graphic circuit gives an inclusion-minimal stopping set.
\end{proof}

The proposition reduces coatom enumeration to the standard problem of
listing graphic circuits.

\begin{corollary}
\label{cor:freq2-enum}
Restricted to hypergraphs $\cH$ with $\Delta(\cH)\le2$, 
\HHEnum{} admits output-linear total-time enumeration. 
More precisely, if outputs are represented by their sparse complements, 
all coatoms can be generated in
$O\bigl(L(\cH)+\sum_{S\in\MinStop(\cH)}(|S|+1)\bigr)$
total time. If coatoms are represented by their $|V|$-bit characteristic vectors, 
the total running time is $O(L(\cH)+N|V|)$, where $N=|\MinStop(\cH)|$ 
is the number of outputs. Moreover, \HHEnum{} belongs to $\DelayP$ 
on this class and admits a polynomial-space algorithm.
\end{corollary}

\begin{proof}
Construct the apex multigraph $Q_{\cH}$ in $O(L(\cH))$ time. 
Output every loop at the apex directly and delete these loops. 
Subdivide every remaining labeled edge $e=uv$ by a distinct new vertex $s_e$, 
replacing it by the path $u\,s_e\,v$. The resulting graph $G$ is simple and has size $O(L(\cH))$.

The simple cycles of $G$ are in bijection with the nonloop graphic circuits of $Q_{\cH}$. 
A circuit $S$ corresponds to a cycle of length $2|S|$, and the subdivision 
vertices identify its original labeled edges uniquely. 
Applying the optimal cycle-listing algorithm of \citet{BirmeleEtAl2013} to 
each edge-containing connected component of $G$ lists all such cycles in total time
$O\bigl(L(\cH)+\sum_{S\in\MinStop(\cH)}|S|\bigr)$.
Decoding each cycle and outputting the apex loops directly give the stated sparse-output bound.

For bit-vector outputs, converting each minimal stopping set into the characteristic 
vector of its complementary coatom costs $\Theta(|V|)$ time. Since every circuit has size at most $|V|$, 
the cycle-listing and decoding costs are absorbed by $O(N|V|)$, giving total time $O(L(\cH)+N|V|)$.

The polynomial-delay and polynomial-space guarantees follow from the direct implementation 
given in Appendix~\ref{app:cycle-delay}.
\end{proof}

This positive result is specific to unconstrained enumeration.
The extension problem remains hard even when every element has
frequency exactly two.

\begin{theorem}
\label{thm:freq2-extension-hard}
The problem \HHEx{} is $\NP$-complete on three-uniform
hypergraphs in which every ground element has frequency exactly two.
The hardness holds with $P=\varnothing$.
\end{theorem}

\begin{proof}
Membership in $\NP$ follows from
\cref{lem:coatom-verification}.

For hardness, reduce from Hamiltonian Cycle in planar cubic
$3$-connected graphs, which is $\NP$-complete
\citep{GareyJohnsonTarjan1976}.
Given such a graph $G$, 
apply the triangle expansion described in Appendix~\ref{app:freq2-extension}.
It produces a cubic graph $G^\triangle$ and a distinguished edge
$f_v$ for every $v\in V(G)$ such that $G$ has a Hamiltonian cycle if
and only if $G^\triangle$ has a simple cycle containing every
distinguished edge; see \cref{lem:triangle-expansion}.

Construct a hypergraph whose ground set is
$E(G^\triangle)$.
For every vertex $u\in V(G^\triangle)$, add the hyperedge consisting
of the three graph edges incident with $u$.
Because $G^\triangle$ is cubic, every hyperedge has size three.
Every ground element is a graph edge and therefore belongs to exactly
the two hyperedges corresponding to its endpoints.

The apex multigraph of this hypergraph is
$G^\triangle$ together with an isolated apex.
By \cref{prop:freq2-structure}, its minimal nonempty stopping sets are
exactly the simple cycles of $G^\triangle$.
Set $P=\varnothing$ and
$N:=\{f_v:v\in V(G)\}$.
By complement duality, an extending coatom exists exactly when there
is a minimal stopping set containing $N$.
This is equivalent to the existence of a simple cycle of
$G^\triangle$ containing every distinguished edge, and hence to the
existence of a Hamiltonian cycle in $G$.
The reduction is polynomial.
\end{proof}

The preceding results give a sharp separation between enumeration and
extension on the same restricted input class.

\begin{corollary}
\label{cor:freq2-same-class}
For three-uniform hypergraphs of exact element frequency two, 
\HHEnum{} admits output-linear total-time enumeration as well as a polynomial-delay, 
polynomial-space algorithm, whereas \HHEx{} is $\NP$-complete.
\end{corollary}

\begin{proof}
The enumeration statement follows from
\cref{cor:freq2-enum}, and the extension statement is
\cref{thm:freq2-extension-hard}.
\end{proof}

In Tanner-graph terms, the three-uniform exact-frequency-two 
class consists of incidence graphs in which every check node has 
degree three and every variable node has degree two. 
Thus complete coatom enumeration has output-linear total time on this class 
and also admits a polynomial-delay, polynomial-space implementation, 
whereas constrained extension is $\NP$-complete.

\subsection{Three-uniform hardness at maximum frequency three}
\label{sec:uniformity}

The lower bound in \cref{cor:rank3-freq3} applies to hypergraphs whose
hyperedges have size two or three and whose maximum element frequency
is at most three.
It remains to consider the case in which every hyperedge has size
exactly three.
We show that every two-element hyperedge can be replaced by four
three-element hyperedges.
This preserves the original minimal stopping sets and introduces
three known minimal stopping sets using only new elements.

Let $p,q$ be two elements, and let
$J:=\{a,b,c,d,f\}$ consist of five fresh elements.
Define
$\mathcal W(p,q):=
\{\{p,q,a\},\{a,d,f\},\{b,c,d\},\{b,c,f\}\}$.
Also let $D_1:=\{b,c\}$, $D_2:=\{b,d,f\}$, and
$D_3:=\{c,d,f\}$.

\begin{lemma}
\label{lem:pair-edge-replacement}
Let $S$ be a stopping set of $\mathcal W(p,q)$.
If $S\cap J=\varnothing$, then $p\in S$ if and only if $q\in S$.
If $S\cap J\ne\varnothing$, then $D_i\subseteq S$ for some
$i\in\{1,2,3\}$.
Consequently, the minimal nonempty stopping sets of
$\mathcal W(p,q)$ are exactly
$\{p,q\}$, $D_1$, $D_2$, and $D_3$.
\end{lemma}

\begin{proof}
If $S\cap J=\varnothing$, then the hyperedge $\{p,q,a\}$ implies that
$p$ and $q$ have the same membership in $S$.

Suppose that $S\cap J\ne\varnothing$, and identify each element with
its membership bit in $S$.
If $a=0$, then the hyperedge $\{a,d,f\}$ gives $d=f$.
If $d=f=0$, the last two triples force $b=c$. 
Since $S\cap J\ne\varnothing$ and $a=d=f=0$, this common value must be one, 
and hence $D_1\subseteq S$.
If $d=f=1$, they require at least one of $b,c$, so $S$ contains
$D_2$ or $D_3$.

If $a=1$, then at least one of $d,f$ belongs to $S$.
If exactly one belongs to $S$, the last two triples force $b=c=1$,
so $D_1\subseteq S$.
If both belong to $S$, at least one of $b,c$ belongs to $S$, so
$S$ contains $D_2$ or $D_3$.

The sets $\{p,q\},D_1,D_2,D_3$ are stopping sets, and a direct check
of their nonempty proper subsets shows that they are minimal.
The preceding argument shows that every other nonempty stopping set
contains one of these four sets.
\end{proof}

Let $\cH$ be a normalized hypergraph on $V$ whose hyperedges have
size two or three, and let
$E_i:=\{A\in\cH:|A|=i\}$ for $i\in\{2,3\}$.
For each hyperedge $e=\{p_e,q_e\}\in E_2$, introduce a fresh set
$J_e:=\{a_e,b_e,c_e,d_e,f_e\}$, delete $e$, and add
$\{p_e,q_e,a_e\}$, $\{a_e,d_e,f_e\}$,
$\{b_e,c_e,d_e\}$, and $\{b_e,c_e,f_e\}$.
The sets $J_e$ are chosen pairwise disjoint.
Retain all hyperedges in $E_3$, and denote the resulting hypergraph
by $\cH^\triangle$.
For each $e\in E_2$, let
$D_{e,1}:=\{b_e,c_e\}$,
$D_{e,2}:=\{b_e,d_e,f_e\}$, and
$D_{e,3}:=\{c_e,d_e,f_e\}$.
Write
$\mathcal D(\cH):=
\{D_{e,i}:e\in E_2,\ i\in\{1,2,3\}\}$.

\begin{theorem}
\label{thm:three-uniform-replacement}
The hypergraph $\cH^\triangle$ is normalized and three-uniform, and
\begin{equation}
  \operatorname{MinStop}(\cH^\triangle)
  =
  \operatorname{MinStop}(\cH)
  \mathbin{\dot\cup}
  \mathcal D(\cH).
  \label{eq:three-uniform-minstop}
\end{equation}
Here the members of $\operatorname{MinStop}(\cH)$ are viewed as
subsets of the enlarged ground set.

If $m_i:=|E_i|$, then $\cH^\triangle$ has
$|V|+5m_2$ elements and $4m_2+m_3$ hyperedges, and
$L(\cH^\triangle)=|V|+21m_2+4m_3=O(L(\cH))$.
The frequency of every element of $V$ is unchanged, while every new
element has frequency two.
Consequently,
$\Delta(\cH^\triangle)\le\max\{\Delta(\cH),2\}$.
The construction takes $O(L(\cH))$ time.
\end{theorem}

\begin{proof}
Consider first a set $S\subseteq V$, so that $S$ contains no new
element.
For each $e=\{p_e,q_e\}\in E_2$, the hyperedge
$\{p_e,q_e,a_e\}$ meets $S$ in exactly $|e\cap S|$ elements, while
the other three hyperedges associated with $e$ are disjoint from
$S$.
The hyperedges in $E_3$ are unchanged.
Therefore,
$S\in\operatorname{Stop}(\cH^\triangle)$ if and only if
$S\in\operatorname{Stop}(\cH)$.
By \cref{lem:pair-edge-replacement}, each $D_{e,i}$ is a minimal
stopping set for the four hyperedges associated with $e$.
It is disjoint from every retained hyperedge and from every
hyperedge associated with a different member of $E_2$.
Thus it is also a minimal stopping set of $\cH^\triangle$.

Let $T$ be a minimal nonempty stopping set of $\cH^\triangle$.
Suppose that $T\cap J_e\ne\varnothing$ for some $e\in E_2$.
Then
$T\cap(\{p_e,q_e\}\cup J_e)$ is a stopping set of
$\mathcal W(p_e,q_e)$ with nonempty intersection with $J_e$.
By \cref{lem:pair-edge-replacement}, it contains some $D_{e,i}$.
Since $D_{e,i}$ is itself a nonempty stopping set of
$\cH^\triangle$, the minimality of $T$ implies $T=D_{e,i}$.
If $T\cap J_e=\varnothing$ for every $e\in E_2$, then $T\subseteq V$.
The equivalence
$S\in\operatorname{Stop}(\cH^\triangle)$ if and only if
$S\in\operatorname{Stop}(\cH)$ for $S\subseteq V$ shows that
$T$ is a stopping set of $\cH$.
If $T$ were not minimal in $\cH$, there would be a nonempty set
$S\subsetneq T$ with $S\in\operatorname{Stop}(\cH)$.
Since $S\subseteq V$, the same equivalence would give
$S\in\operatorname{Stop}(\cH^\triangle)$, contradicting the
minimality of $T$.
Therefore $T\in\operatorname{MinStop}(\cH)$.

Conversely, let $S\in\operatorname{MinStop}(\cH)$.
Since $S\subseteq V$, the same equivalence gives
$S\in\operatorname{Stop}(\cH^\triangle)$.
If a nonempty proper subset $R\subsetneq S$ were a stopping set of
$\cH^\triangle$, then $R\subseteq V$ and hence
$R\in\operatorname{Stop}(\cH)$, contradicting the minimality of $S$.
Thus $S\in\operatorname{MinStop}(\cH^\triangle)$.
This proves \eqref{eq:three-uniform-minstop}.

Each replaced two-element hyperedge introduces five elements and four
three-element hyperedges, which gives the stated size formulas.
Because the new elements are fresh, no duplicate hyperedge is
introduced.
Each occurrence of an original element in a two-element hyperedge is
replaced by one occurrence in a three-element hyperedge, while all
other occurrences are unchanged.
Thus every original frequency is preserved.
Each new element occurs in exactly two hyperedges.
\end{proof}

We now apply \cref{thm:three-uniform-replacement} to the hard instances
used in \cref{cor:rank3-freq3}.
Since the replacement does not increase maximum element frequency, we
obtain the following.

\begin{corollary}
\label{cor:three-uniform-freq3}
Unless $\mathsf P=\mathsf{NP}$, \HHEnum{} is not in $\OutputP$ for
three-uniform hypergraphs of maximum element frequency at most three.
\end{corollary}

\begin{proof}
By \cref{prop:complement-duality}, it is enough to consider minimal
stopping sets.
Let $\cH$ be an instance from \cref{cor:rank3-freq3}, and construct
$\cH^\triangle$.
By \cref{thm:three-uniform-replacement}, the resulting hypergraph is
three-uniform, has maximum element frequency at most three, and has
incidence length linear in $L(\cH)$.

Suppose that minimal stopping sets of such hypergraphs could be
enumerated in output-polynomial time.
Apply the algorithm to $\cH^\triangle$, discard the sets in
$\mathcal D(\cH)$, and output every other set.
By \eqref{eq:three-uniform-minstop}, this enumerates exactly
$\operatorname{MinStop}(\cH)$.

Only $3|E_2|$ outputs are discarded, and the enlarged ground set has
size $|V|+5|E_2|$.
Under the bit-vector output convention, the input and output length
of the new instance is polynomially bounded by the input and output
length of the original instance.
This would give an output-polynomial algorithm for the instances of
\cref{cor:rank3-freq3}, a contradiction.
\end{proof}

%
%
%

The same replacement also transfers the incremental lower bound.

\begin{corollary}
\label{cor:three-uniform-incremental}
The problem \IncHH{} is $\NP$-complete for three-uniform hypergraphs 
of maximum element frequency at most three. 
The hardness also holds for the promise version.
\end{corollary}

\begin{proof}
Given an instance $(\cH,\mathcal L)$ covered by
\cref{cor:rank3-freq3}, first consider the unrestricted version.
Using \cref{lem:coatom-verification}, test in polynomial time whether
every member of $\mathcal L$ is a coatom.
If the list is invalid, map the input to the following fixed no-instance.
Let $\cH_0$ have ground set $\{a,b,c\}$ and the single hyperedge
$\{a,b,c\}$, and let $\mathcal L_0=(\varnothing)$.
The hypergraph $\cH_0$ is three-uniform and has maximum element
frequency one, while $\varnothing$ is not a coatom of
$\Phi_{\cH_0}$.
Hence $(\cH_0,\mathcal L_0)$ is a no-instance of
\IncHH{} in the required class.
This validity test and the fixed-instance branch are unnecessary for
the promise version.

Suppose now that every member of $\mathcal L$ is a coatom.
Construct $\cH^\triangle$ as in
\cref{thm:three-uniform-replacement}.
For a source coatom $M=V\setminus S$, define
\(
M^\triangle
  :=V(\cH^\triangle)\setminus S
   =M\cup\bigcup_{e\in E_2}J_e .
\)
By \cref{thm:three-uniform-replacement}, this maps the coatoms of
$\Phi_{\cH}$ bijectively onto the coatoms of
$\Phi_{\cH^\triangle}$ whose complementary minimal stopping sets lie
in $\MinStop(\cH)$.

Map every member of $\mathcal L$ by
$M\mapsto M^\triangle$, and append the coatoms
$V(\cH^\triangle)\setminus D_{e,i}$ for all
$e\in E_2$ and $i\in\{1,2,3\}$.
By \cref{eq:three-uniform-minstop}, the resulting list omits a coatom
of $\cH^\triangle$ if and only if the original list omits a coatom of
$\cH$.
The construction is polynomial, and membership in $\NP$ holds in
general.
\end{proof}


In Tanner-graph terms, the final hard instances have check-node 
degree exactly three and variable-node degree at most three. 
Unless $\mathsf P=\mathsf{NP}$, complete coatom enumeration 
is not in $\OutputP$ on this class. Moreover, \IncHH{} is $\NP$-complete, 
and its hardness persists when the input family is promised to contain only coatoms. 
Together with the frequency-two result above, this gives, 
assuming $\mathsf P\ne\mathsf{NP}$, the exact threshold between maximum 
variable-node degrees two and three for arbitrary-order enumeration 
when every check node has degree three.

\section{Conclusion}
\label{sec:conclusion}

For every normalized Horn CNF $\Gamma$, we have constructed in linear time a
hypergraph $\cC(\Gamma)$ of rank at most three such that
\(
  \Mod(\Gamma)
  \cong
  \Mod(\Phi_{\cC(\Gamma)})\setminus\{U_\Gamma\}
\)
under inclusion.
Each source model $M$ has a unique proper target model $K$ satisfying 
$K\cap X=M$, and maximal models of $\Gamma$ correspond 
bijectively to the coatoms in $\cM(\Phi_{\cC(\Gamma)})$.

This representation transfers the maximal-Horn-model lower bound to
\HHEnum{}.
Unless $\mathsf P=\mathsf{NP}$, the problem is not in $\OutputP$ even
when $|A|\in\{2,3\}$ for every input hyperedge $A$.
The incidence-splitting transformation preserves the inclusion order
of all stopping sets and reduces maximum element frequency to three.
Consequently, the same lower bound holds simultaneously under
$|A|\in\{2,3\}$ and $\Delta(\cH)\le3$.
The adjacent positive cases show that these bounds are conditionally
sharp for arbitrary-order enumeration.
If $\rank(\cH)\le2$, then \HHEnum{} admits output-linear total time
after normalization.
If $\Delta(\cH)\le2$, then \HHEnum{} admits output-linear total-time enumeration 
through the graphic-circuit representation of $\MinStop(\cH)$; 
it also belongs to $\DelayP$ and admits a polynomial-space algorithm.

The complexity changes when additional requirements are imposed on
the output.
A member of $\cM(\Phi_{\cH})$ can be found,
or its nonexistence can be detected, in polynomial time,
whereas deciding whether there exists
$M\in\cM(\Phi_{\cH})$ with $|V\setminus M|\le k$ is
$\NP$-complete.
The problems \IncHH{} and \HHEx{} are also $\NP$-complete.
In particular, on three-uniform hypergraphs with $\Delta(\cH)=2$, 
\HHEnum{} admits output-linear total-time enumeration and 
also belongs to $\DelayP$ with polynomial space, 
while \HHEx{} remains $\NP$-complete.

Replacing every remaining two-element hyperedge by the local
three-uniform construction preserves all original minimal stopping
sets and introduces only a polynomial number of explicitly known
additional outputs.
Consequently, unless $\mathsf P=\mathsf{NP}$, \HHEnum{} is not in
$\OutputP$ even for three-uniform hypergraphs of maximum element
frequency at most three.

In Tanner-graph terms, the final hard instances have check-node degree
exactly three and variable-node degree at most three.
By contrast, variable-node degree at most two admits output-linear total-time 
enumeration and also a polynomial-delay, polynomial-space implementation, 
even when every check node has degree three.
Thus, assuming $\mathsf P\ne\mathsf{NP}$, the threshold for arbitrary-order 
enumeration on check-degree-three Tanner graphs lies exactly between 
maximum variable-node degrees two and three. Constrained extension 
behaves differently: it remains $\NP$-complete already for $(2,3)$-regular Tanner graphs.

The present three-uniform constructions are not linear, 
since some pairs of hyperedges share two elements 
and create $4$-cycles in the Tanner graph. 
The present results leave open whether the same lower bound persists 
for linear three-uniform hypergraphs, equivalently for Tanner graphs 
of girth at least six, and whether it extends to $(3,3)$-regular Tanner graphs.

\appendix

\section{A Horn formulation of the KSS hard family}
\label{app:kss}

The lower bound for maximal Horn-model enumeration used in this paper
is due to \citet{KavvadiasSideriStavropoulos2000}.
For completeness, this appendix restates 
the positive one-in-three reduction underlying their result 
in the notation used here and verifies the properties 
used throughout \cref{sec:hardness-transfer}.
The source problem is \textsc{Positive One-in-Three 3SAT}, which is $\NP$-complete \citep{Schaefer1978}.
Let $F=\bigwedge_{i=1}^{m}C_i$, where
$C_i=\{x_{i,1},x_{i,2},x_{i,3}\}$ and the three variables in each
clause are distinct.
Writing $X_F$ for the variable set of $F$, an assignment
$M\subseteq X_F$ satisfies $F$ if
$|M\cap C_i|=1$ for every $i\in[m]$.
We delete variables occurring in no clause.
This does not change satisfiability and ensures that every source
variable occurs in at least one clause.

For each clause $C_i$, introduce a variable $y_i$, and introduce one
additional variable $z$.
The Horn CNF $\Gamma_F$ contains the following clauses for every
$i\in[m]$:
the three pair clauses
$\{x_{i,1},x_{i,2}\}\to z$,
$\{x_{i,2},x_{i,3}\}\to z$, and
$\{x_{i,3},x_{i,1}\}\to z$;
the three gate clauses
$\{z,x_{i,j}\}\to y_i$ for $j\in[3]$;
and the back clause $\{y_i\}\to z$.
Finally, it contains the negative clause
$\{y_1,\ldots,y_m,z\}\to\bot$.
All clause premises are nonempty.

For each $i\in[m]$, define
$T_i:=\{z\}\cup\{y_j:j\ne i\}\cup(X_F\setminus C_i)$.
Thus $T_i$ sets $y_i$ and the three variables of $C_i$ to false and
sets every other variable to true.

\begin{lemma}
\label{lem:kss-z1}
The maximal models of $\Gamma_F$ containing $z$ are exactly
$T_1,\ldots,T_m$.
\end{lemma}

\begin{proof}
Let $T$ be a model containing $z$.
The negative clause forces $y_i\notin T$ for some $i$.
The three gate clauses for $C_i$ then force every variable of $C_i$
to be absent from $T$, so $T\subseteq T_i$.
Since $T_i$ is a model, maximality of $T$ implies $T=T_i$.

Conversely, $T_i$ satisfies every clause.
A strict extension must add either $y_i$ or a variable of $C_i$.
Adding $y_i$ violates the negative clause.
Adding a variable of $C_i$ while keeping $y_i$ false violates the
corresponding gate clause, whereas also adding $y_i$ again violates
the negative clause.
Hence $T_i$ is maximal.
\end{proof}

\begin{lemma}
\label{lem:kss-z0}
The maximal models of $\Gamma_F$ omitting $z$ are in bijection with
the one-in-three satisfying assignments of $F$.
Every such maximal model also omits all variables $y_i$.
\end{lemma}

\begin{proof}
Let $T$ be a model omitting $z$.
The back clauses force every $y_i$ to be absent, and the pair clauses
imply that at most one source variable is present in each clause.
If $T$ is maximal and some clause $C_i$ contains no present source
variable, then $T\subsetneq T_i$, contradicting maximality.
Thus exactly one source variable is present in each clause.

Conversely, let $M\subseteq X_F$ be a one-in-three satisfying
assignment and extend it by keeping $z$ and all $y_i$ false.
The resulting assignment satisfies $\Gamma_F$.
If a strict extension still omits $z$, then adding some $y_i$ violates
the corresponding back clause, while adding a previously false source
variable creates a clause with two true variables and violates a pair
clause.
If the extension contains $z$, the unique true source variable in
each clause forces every $y_i$ to be true, after which the negative
clause is violated.
Hence the model is maximal.
\end{proof}

\begin{corollary}
\label{cor:kss-family}
The formula $\Gamma_F$ and the family
$\mathcal L_F:=\{T_i:i\in[m]\}$ can be constructed in polynomial time,
and every clause of $\Gamma_F$ has a nonempty premise.
Moreover,
\[
  \MaxMod(\Gamma_F)
  =
  \mathcal L_F
  \mathbin{\dot\cup}
  \bigl\{
    M\subseteq X_F:
    |M\cap C_i|=1\text{ for every }i\in[m]
  \bigr\}.
\]
Consequently,
$\MaxMod(\Gamma_F)=\mathcal L_F$ if and only if $F$ is a no-instance,
whereas $F$ is a yes-instance if and only if
$\Gamma_F$ has a maximal model omitting $z$.
On a no-instance, the complete maximal-model output is the explicitly
known family $\mathcal L_F$ of size $m$.
\end{corollary}

\begin{proof}
The description of the maximal models containing $z$ follows from
\cref{lem:kss-z1}, and the description of those omitting $z$ follows
from \cref{lem:kss-z0}.
The remaining statements follow directly from the construction.
\end{proof}

\section{Triangle expansion for exact frequency two}
\label{app:freq2-extension}

This appendix proves the graph transformation used in
\cref{thm:freq2-extension-hard}.
Hamiltonian Cycle is $\NP$-complete for planar cubic $3$-connected
graphs \citep{GareyJohnsonTarjan1976}; such graphs are simple and have
at least four vertices.

Let $G$ be a cubic graph.
For each vertex $v$, replace $v$ by a triangle with ports
$v_1,v_2,v_3$, and attach the three former incident edges bijectively
to these ports.
Denote the resulting cubic graph by $G^\triangle$ and distinguish the
triangle edge $f_v:=v_1v_2$.

\begin{lemma}
\label{lem:triangle-expansion}
The graph $G$ has a Hamiltonian cycle if and only if $G^\triangle$
has a simple cycle containing every distinguished edge $f_v$.
For fixed port assignments and distinguished edges, the two cycles
determine each other uniquely.
\end{lemma}

\begin{proof}
Let $D$ be a Hamiltonian cycle of $G$.
At each source vertex it uses two ports.
Replace that vertex by the unique path in its triangle joining those
ports and containing $f_v$: use $v_1v_2$ for ports
$\{v_1,v_2\}$, $v_1v_2v_3$ for $\{v_1,v_3\}$, and
$v_2v_1v_3$ for $\{v_2,v_3\}$.
Together with the external edges of $D$, these paths form a simple
cycle of $G^\triangle$ containing every $f_v$.

Conversely, let $C$ be a simple cycle of $G^\triangle$ containing all
$f_v$.
Inside one triangle, $C$ cannot use all three triangle edges: then all
three ports would already have degree two in $C$, so that triangle
would be the entire cycle, contradicting the presence of the
distinguished edges in the other gadgets.
Hence the intersection of $C$ with each triangle is one of the three
paths listed above, and exactly two external edges of the gadget lie
on $C$.
Contracting every such path gives a connected spanning $2$-regular
subgraph of $G$, hence a Hamiltonian cycle.
The local path is unique for each pair of used ports, so the two
transformations are inverse.
\end{proof}

To obtain the hypergraph used in the reduction, take
$E(G^\triangle)$ as the ground set and, for every
$u\in V(G^\triangle)$, add the hyperedge
$\delta(u)$ consisting of the three graph edges incident with $u$.
The hypergraph is three-uniform, and every ground element has
frequency exactly two.
Its apex multigraph is $G^\triangle$ together with an isolated apex,
so \cref{prop:freq2-structure} identifies its minimal nonempty
stopping sets with the simple cycles of $G^\triangle$.
With $P=\varnothing$ and
$N=\{f_v:v\in V(G)\}$, complement duality shows that an extending
coatom exists exactly when $G^\triangle$ has a simple cycle containing
all distinguished edges.
By \cref{lem:triangle-expansion}, this is equivalent to a Hamiltonian
cycle in $G$.
The reduction has linear size.

\section{Polynomial delay for the frequency-two enumerator}
\label{app:cycle-delay}

This appendix supplies the direct polynomial-delay 
and polynomial-space implementation used in \cref{cor:freq2-enum}. 
The output-linear total-time bound in that corollary follows separately 
from the optimal cycle-listing algorithm of \citet{BirmeleEtAl2013}.

Let $Q_{\cH}$ be the apex multigraph from \cref{sec:freq2}.
Output its apex loops directly, delete them, and subdivide every
remaining labeled edge $e=uv$ by a distinct new vertex $s_e$.
The resulting graph $G$ is simple, and its simple cycles are in
bijection with the nonloop graphic circuits of $Q_{\cH}$.

Fix a total order on $V(G)$.
Every simple cycle has a unique smallest vertex $s$ and a unique
unordered pair $\{a,b\}$ of its two neighbors on the cycle.
Both neighbors are larger than $s$, and deleting $s$ leaves a simple
$a$--$b$ path in the subgraph induced by the vertices larger than
$s$.
We therefore process all tasks $(s,\{a,b\})$ with
$a,b\in N_G(s)$ and $s<a,b$.
For one task, orient the pair by $a<b$. 
First test whether $b$ is reachable from $a$ in $G[\{v:v>s\}]$ and 
skip the task if it is not. Otherwise, enumerate the simple 
$a$--$b$ paths in this induced subgraph by depth-first search.

A search state is a simple path $P$ from $a$ to a current endpoint
$v$.
For a neighbor $u\notin V(P)$, output if $u=b$; otherwise recurse only
when $b$ is reachable from $u$ after deleting all vertices of $P$.
The reachability test is necessary and sufficient for a simple
continuation, so every retained search node has an output descendant.
The minimum vertex and its two cycle neighbors determine the task
uniquely, and the orientation $a<b$ prevents reversal duplicates.
Thus every simple cycle is output exactly once.

Let $n=|V(G)|$ and $m=|E(G)|$.
The search depth is at most $n$. Between two consecutive outputs within 
a nonempty task, the depth-first search backtracks through at most $n$ levels 
and tests at most $n$ candidate neighbors at each level. 
Since one reachability test costs $O(n+m)$ time, the delay within a 
nonempty task is $O(n^2(n+m))$. An empty task is rejected by the initial reachability test,
and the number of tasks is at most
$\sum_s\binom{\deg_G(s)}2=O(n^3)$.
Therefore the delay before the first output, between consecutive
outputs, and after the last output is
$O(n^3(n+m))$.

The recursion stack, path marks, adjacency lists, and reusable
reachability workspace require $O(n+m)$ space.
Decoding the subdivision vertices recovers the corresponding circuit
of $Q_{\cH}$, and complementing its labels gives the coatom.

\section{A direct embedding of minimal-transversal enumeration}
\label{app:transversal}

This appendix gives a linear-size, 
output-bijective reduction from minimal-transversal enumeration to \HHEnum{}. 
Unlike the classical key--antikey blocker identity, 
the construction applies to an arbitrary explicit source hypergraph 
and does not require realizability as a minimal-key clutter. 
It is not used in the non-$\OutputP$ proof of \cref{thm:rank3-hardness}; 
its purpose is to record an exact correspondence between 
the two output families and, in particular, 
to transfer upper bounds for \HHEnum{} to minimal-transversal enumeration.

Minimal-transversal enumeration admits an
output-quasipolynomial algorithm, while the existence of an
output-polynomial algorithm remains open
\citep{FredmanKhachiyan1996,EiterGottlob1995,
EiterMakinoGottlob2008}.
Unlike the representation theorem of \cref{thm:compiler}, the
construction below does not first encode a Horn formula.
It translates the transversal condition directly into a stopping-set
condition, with one nonempty stopping set for every source transversal
and no additional nonempty stopping sets.

Let $\cA\subseteq 2^X$ be a hypergraph.
Duplicate source hyperedges may be removed, 
but an empty source hyperedge, if present, 
is retained because it rules out every transversal. 
In this appendix, we use the incidence length 
$L(\cA):=|X|+|\cA|+\sum_{E\in\cA}|E|$ with the empty hyperedge retained.
A set $T\subseteq X$ is a \emph{transversal} of $\cA$ if
$T\cap E\ne\varnothing$ for every $E\in\cA$.
We denote the family of inclusion-minimal transversals by $\Tr(\cA)$.

Introduce a root element $r$ and, for every $x\in X$, a new element
$\bar x$.
The ground set of the target hypergraph $\cH_{\cA}$ is
$U_{\cA}:=X\cup\{r\}\cup\{\bar x:x\in X\}$.
For each $x\in X$, add the two guard hyperedges
$\{r,\bar x\}$ and $\{r,x,\bar x\}$.
For each source hyperedge $E\in\cA$, add the hyperedge
$\{r\}\cup E$.
Finally, let
$B_{\cA}:=\{r\}\cup\{\bar x:x\in X\}$.

The guards make the all-zero assignment the only stopping set with
$r=0$.
Once $r=1$, they force every $\bar x$ to be selected while leaving
$x$ unrestricted.
The hyperedge $\{r\}\cup E$ then requires the selected source
coordinates to meet $E$.

\begin{theorem}
\label{thm:transversal}
The map $T\mapsto B_{\cA}\cup T$ is an inclusion-order isomorphism
from the family of all transversals of $\cA$ onto
$\Stop(\cH_{\cA})\setminus\{\varnothing\}$.
Consequently,
$\MinStop(\cH_{\cA})=\{B_{\cA}\cup T:T\in\Tr(\cA)\}$.
Under complement duality, the map
$T\mapsto U_{\cA}\setminus(B_{\cA}\cup T)=X\setminus T$
is a bijection from $\Tr(\cA)$ onto the coatom family of
$\Phi_{\cH_{\cA}}$.

The construction has incidence length $O(L(\cA))$ and can be carried
out in $O(L(\cA))$ time.
For bit-vector outputs, the forward and inverse maps are computable in
$O(L(\cA))$ time per output.
\end{theorem}

\begin{proof}
Consider the two guard hyperedges associated with one element
$x\in X$.
The hyperedge $\{r,\bar x\}$ forces the membership bits of $r$ and
$\bar x$ to agree.
If $r=0$, then $\bar x=0$, and the hyperedge
$\{r,x,\bar x\}$ forces $x=0$.
Since this holds for every $x\in X$, the empty set is the only
stopping set with $r=0$.

Let $S$ now be a nonempty stopping set of $\cH_{\cA}$.
Then $r\in S$, and the guard edges force $\bar x\in S$ for every
$x\in X$.
With these two elements selected, the hyperedge
$\{r,x,\bar x\}$ meets $S$ in either two or three elements and
imposes no restriction on whether $x$ belongs to $S$.
Put $T:=S\cap X$.
For every $E\in\cA$, we have
$|(\{r\}\cup E)\cap S|=1+|E\cap T|$.
This number differs from one exactly when $E\cap T\ne\varnothing$.
Thus $S$ is a nonempty stopping set exactly when
$S=B_{\cA}\cup T$ for a transversal $T$ of $\cA$.
Conversely, every set of this form satisfies all guard hyperedges and
all hyperedges associated with members of $\cA$ whenever $T$ is a
transversal.

Because every nonempty stopping set contains the same fixed block
$B_{\cA}$, the map $T\mapsto B_{\cA}\cup T$ preserves and reflects
inclusion.
It therefore sends the inclusion-minimal transversals exactly to the
inclusion-minimal nonempty stopping sets.
The coatom correspondence follows from
\cref{prop:complement-duality}.

The target has $2|X|+1$ ground elements, $2|X|+|\cA|$ hyperedges,
and $5|X|+\sum_{E\in\cA}(|E|+1)$ incidences.
These quantities are linear in $L(\cA)$, and the construction can be
written down in the same time.
For a target coatom $K$, the inverse output map is
$T=(U_{\cA}\setminus K)\cap X$.
Both output transformations are linear in the target and output
encodings.
\end{proof}

The theorem includes the two boundary cases.
If $\cA=\varnothing$, then $\varnothing$ is the unique minimal
transversal and $B_{\cA}$ is the unique minimal nonempty stopping
set.
If $\varnothing\in\cA$, then $\cH_{\cA}$ contains the singleton
hyperedge $\{r\}$, so it has no nonempty stopping set, in agreement
with the fact that $\cA$ has no transversal.

\begin{corollary}
\label{cor:transversal-reduction}
Explicit minimal-transversal enumeration reduces to \HHEnum{} by a
linear-size, output-bijective reduction.
\end{corollary}

\begin{proof}
Apply \cref{thm:transversal} and use the mutually inverse output maps
$T\mapsto X\setminus T$ and
$K\mapsto(U_{\cA}\setminus K)\cap X$.
\end{proof}

The direction of the reduction is
$\textsc{Trans-Enum}\le\HHEnum{}$.
Thus a lower bound for minimal-transversal enumeration transfers to
\HHEnum{}, while an upper bound for \HHEnum{} transfers to
minimal-transversal enumeration.
The converse implications do not follow.
In particular, the output-quasipolynomial algorithm for
minimal-transversal enumeration
\citep{FredmanKhachiyan1996} does not yield an algorithm with the same
guarantee for \HHEnum{}.
The corollary therefore records an exact containment of output
problems rather than an additional non-$\OutputP$ lower bound.

The construction is different from the classical key--antikey blocker
identity.
For a closure system, the complements of its coatoms are the minimal
transversals of its minimal-key clutter
\citep{Thi1986,DemetrovicsThi1987,Son2006}.
Using this identity algorithmically requires the minimal-key clutter
to be explicitly available.
The construction above instead starts from an arbitrary explicitly
given hypergraph $\cA$ and encodes its transversal condition directly.
It does not require $\cA$ to be realizable as the minimal-key family
of a hypergraph Horn function; key realization for hypergraph Horn
functions is studied by \citet{BercziBorosMakino2024}.

The embedding concerns the unrestricted enumeration problem.
A source hyperedge $E$ becomes a target hyperedge of size $|E|+1$,
and the root belongs to every guard hyperedge and every hyperedge
associated with a member of $\cA$.
The construction therefore does not preserve bounded rank or bounded
element frequency.
The bounded-parameter results in \cref{sec:thresholds} use different
transformations.

\bibliographystyle{plainnat}
\bibliography{./coatom_Horn.bib}

\end{document}